\documentclass[10pt]{amsart}
\usepackage{graphicx}
\usepackage{amscd}
\usepackage{amsmath}
\usepackage{amsthm}
\usepackage{amsfonts}
\usepackage{amssymb}
\usepackage{mathrsfs}
\usepackage{bm}
\usepackage{enumerate}
\usepackage{amsrefs}
\usepackage{xcolor}
\usepackage[colorlinks, citecolor=blue, linkcolor=red, pdfstartview=FitB]{hyperref}

\usepackage{url}						%
\usepackage[utf8,latin1]{inputenc}	%
\usepackage[english]{babel}	%

\theoremstyle{plain}
\newtheorem{theorem}{Theorem}[section]
\newtheorem{lemma}[theorem]{Lemma}

\newtheorem{corollary}[theorem]{Corollary}

\theoremstyle{definition}

\newtheorem{example}[theorem]{Example}

\theoremstyle{remark}

\newcommand{\mc}[1]{\mathcal{#1}}
\newcommand{\frk}[1]{\mathfrak{#1}}
\newcommand{\CC}{\mathbb{C}}
\newcommand{\DD}{\mathbb{D}}

\newcommand{\NN}{\mathbb{N}}
\newcommand{\gra}{\alpha}

\newcommand{\grs}{\sigma}
\newcommand{\grG}{\Gamma}

\newcommand{\bksl}{\backslash}
\newcommand{\nin}{\notin}
\newcommand{\cc}[1]{\overline{#1}}
\DeclareMathOperator{\ran}{ran}

\newcommand{\pd}{\partial}
\newcommand{\of}{\circ}

\newcommand{\BB}{\mathbb{B}}
\newcommand{\ip}[1]{\langle #1 \rangle}
\newcommand{\qand}{\quad\text{and}\quad}
\newcommand{\ol}[1]{\overline{#1}}
\newcommand{\wtil}[1]{\widetilde{#1}}

\newcommand{\bB}{{\mathbb{B}}} 
\newcommand{\bC}{{\mathbb{C}}} 
\newcommand{\bD}{{\mathbb{D}}}

\newcommand{\bN}{{\mathbb{N}}}

  \newcommand{\A}{{\mathcal{A}}} 
	\newcommand{\B}{{\mathcal{B}}}
  
  \newcommand{\D}{{\mathcal{D}}}
  
  \newcommand{\F}{{\mathcal{F}}}
  
\renewcommand{\H}{{\mathcal{H}}}

  \newcommand{\M}{{\mathcal{M}}}
  \newcommand{\N}{{\mathcal{N}}}
\renewcommand{\O}{{\mathcal{O}}}

\renewcommand{\S}{{\mathcal{S}}}

  \newcommand{\Z}{{\mathcal{Z}}}

\newcommand{\fa}{{\mathfrak{a}}}

\newcommand{\fb}{{\mathfrak{b}}}

\newcommand{\fc}{{\mathfrak{c}}}

\newcommand{\fE}{{\mathfrak{E}}}
\newcommand{\fF}{{\mathfrak{F}}}

\newcommand{\fH}{{\mathfrak{H}}}

\newcommand{\fK}{{\mathfrak{K}}}
\newcommand{\fM}{{\mathfrak{M}}}

\newcommand{\fp}{{\mathfrak{p}}}

\newcommand{\fw}{{\mathfrak{w}}}
\newcommand{\fX}{{\mathfrak{X}}}

\newcommand{\rC}{\mathrm{C}}

\DeclareMathOperator{\Ann}{Ann}											
\DeclareMathOperator{\Char}{Char}										
\DeclareMathOperator{\TSpec}{\sigma_{\mathrm{Ta}}}	
\newcommand{\ZBd}{\mc{Z}_{\BB_d}}										
\newcommand{\ZcBd}{\mc{Z}_{\overline{\BB}_d}}				
\DeclareMathOperator{\supp}{supp}										
\DeclareMathOperator{\spn}{span}

\newcommand{\pSpec}{\sigma_{\mathrm{p}}}						
\newcommand{\wSpec}{\sigma_{\mathrm{w}*}}						
\newcommand{\kZ}[1]{Z^{(#1)}}												
\newcommand{\kZa}{\kZ{\frk{a}}}											
\DeclareMathOperator{\Loc}{Loc}											
\DeclareMathOperator{\AZ}{AZ}												
\newcommand{\MHFC}{\widehat{\alpha}}								
\newcommand{\MH}{\mathcal{M}(\mathcal{H})}					
\newcommand{\AH}{\mathcal{A}(\mathcal{H})}					

\numberwithin{equation}{section}

\title[Localizable points in the support of an ideal and spectra]{Localizable points in the support of a multiplier ideal and spectra of constrained operators}
\author{Rapha\"el Clou\^atre}

\address{Department of Mathematics, University of Manitoba, Winnipeg, Manitoba, Canada R3T 2N2}

\email{raphael.clouatre@umanitoba.ca\vspace{-2ex}}
\thanks{R.C was partially supported by an NSERC Discovery Grant. E.J.T was partially suppored by a PIMS postdoctoral fellowship.}
\author{Edward J. Timko}
\email{edward.timko@umanitoba.ca\vspace{-2ex}}

\begin{document}

\begin{abstract}
    A unitarily invariant, complete Nevanlinna--Pick kernel $K$ on the unit ball determines a class of operators on Hilbert space called $K$-contractions. We study those $K$-contractions that are constrained, in the sense that they are annihilated by an ideal of multipliers. Our overarching goal is to identify various joint spectra of these constrained $K$-contractions through the vanishing locus of their annihilators. Our methods are based around a careful analysis of a subset of the ball associated to the annihilator, which we call its support. 
    For the functional models, we show how this support completely determines several natural joint spectra. The picture is more complicated for general $K$-contractions, as their spectra can be properly contained in the support. Nevertheless, the ``localizable'' portion of the support always consists of spectral points. When the support is assumed to be small in an appropriate sense, we manage to effectively detect points of localizability.
    
\end{abstract}

\maketitle

\section{Introduction} 

The purpose of this paper, in simplest terms, is to explore the extent to which holomorphic constraints on commuting Hilbert space operators capture spectral information. A manifestation of this paradigm is already visible in elementary matrix theory: the spectrum of a square matrix is given as the zero set of its minimal polynomial. We aim to establish infinite-dimensional multivariate analogues of this basic fact. To do so, we take inspiration from the theory of contractions of class $C_0$, which we now briefly review.

Let $\DD\subset\bC$ denote the open unit disc and let $H^\infty(\DD)$ be the Banach algebra of bounded holomorphic functions on $\DD$. Recall that a contraction $A$ on a Hilbert space $\frk{H}$ is said to be \textit{absolutely continuous} if $A$ admits a weak-$*$ continuous functional calculus 
\[
f\mapsto f(A), \quad f\in H^\infty(\DD).
 \]
The contraction $A$ is said to be \textit{constrained} (or \textit{of class $C_0$}) when the weak-$*$ closed ideal $\{f\in H^\infty(\DD):f(A)=0\}$ is non-trivial. In this case, it is a consequence of Beurling's theorem that this ideal is generated by an inner function $\theta$. Associated with this inner function is a subset $\supp(\theta)$ of $\cc{\DD}$ called its \textit{support}, which can be described as follows.
A point $z\in\cc{\DD}$ is in the support of $\theta$ if either $\theta(z)=0$ or  $\theta$ cannot be analytically continued through $z$.
It is shown in \cite[Theorem 7]{SzNagyFoias1964_ContractVII} (see also \cite[Theorem 2.4.11]{bercovici1988}) that the spectrum of $A$ satisfies
\begin{equation}\label{Eq:SpecEqSuppd1}
	\sigma(A)=\supp(\theta).
\end{equation}
Moreover, the zero set of $\theta$ coincides with the point spectrum of $A$.

We wish to generalize \eqref{Eq:SpecEqSuppd1} in two ways. First, we will be taking a multivariate point of view, where single operators are replaced by finite tuples of commuting operators. Such a perspective fits in with our ongoing effort \cite{CTcyc} to develop a generalization of some of the finer points of the classical Sz.-Nagy--Foias theory of contractions \cite{nagy2010}, much in the spirit Popescu's ambitious and robust program (see \cite{popescu2006} and references therein). The appropriate replacement for the algebra $H^\infty(\DD)$ in this case is well known to be the multiplier algebra of the Drury--Arveson space on the unit ball. The second main aspect of our generalization is that we will be working with operators whose structure is encoded by a wider class of algebras of functions. Whereas the Hardy space and the Drury--Arveson space are the natural settings to study contractions and row contractions, we will be considering commuting tuples of operators which are subject to more general conditions. Let us be more precise regarding this particular point.

Throughout, $d\geq 1$ will be a fixed positive integer. We let $\BB_d\subset \CC^d$ denote the open unit ball. We let $\H$ be a Hilbert function space on $\BB_d$ with a regular unitarily invariant reproducing  kernel $K$. More concretely, this means that
\[ K(z,w)=1+\sum_{n=1}^\infty a_n\ip{z,w}^n, \quad z,w\in\BB_d, \]
for some sequence of positive numbers $(a_n)_n$ satisfying
\[ \lim_{n\to\infty}\frac{a_n}{a_{n+1}}=1. \]
We also assume that $K$ satisfies the complete Nevanlinna--Pick  property, which in this context is equivalent to the existence of a uniquely determined non-negative sequence $(b_n)_n$ such that
\[ 1-\frac{1}{K(z,w)}=\sum_{n=1}^n b_n\ip{z,w}^n, \quad z,w\in\BB_d. \]
The class of operators we will be interested in are those commuting $d$-tuples  $T=(T_1,\ldots,T_d)$ on some Hilbert space $\frk{H}$ with the property that 
\[ \sum_{n=1}^\infty b_{n}\sum_{|\alpha|=n}\frac{n!}{\alpha!}T^{\alpha}T^{*\alpha} \leq I; \]
we then say that $T$ is a \emph{$K$-contraction}. The motivation for this choice of terminology becomes clear upon looking at the prototypical example where $\H$ is simply the classical Hardy space on the disc. In this case, 
$
 K(z,w)=(1-\ol{w}z)^{-1}
$
so that $a_n=1$ for every $n\geq 1$, while $b_1=1$ and $b_n=0$ for every $n\geq 2$. Thus, in this context a $K$-contraction is simply a usual contraction. Likewise, when $K$ is the kernel of the Drury--Arveson space, the notion of $K$-contraction reduces to that of a row contraction. But this setup includes several other meaningful function spaces, such as the Dirichlet space on the disc (see Example \ref{E:Hs} for an entire scale of such spaces). Based on foundational work of Agler \cite{agler1982} and significant contributions of Ambrozie--Engli\u s--M\"uller \cite{ambrozie2002}, there is an ongoing effort  to tackle many function spaces and their associated $K$-contractions in one fell swoop \cite{CH2016},\cite{BHM2018}, and this paper is yet another contribution towards that goal. 

The multiplier algebra of $\H$ is denoted by $\M(\H)$. When $K$ is a regular unitarily invariant complete Nevanlinna--Pick kernel, the coordinate functions $x_1,\ldots,x_d$ are multipliers, and we denote by $\AH$ the norm closure of the polynomials inside of $\M(\H)$. An absolutely continuous $K$-contraction $T$ admits a weak-$*$ continuous $\M(\H)$-functional calculus. 
The \emph{annihilator} of $T$ is defined to be the kernel of the associated functional calculus. Unlike the situation described for single contractions of class $C_0$, the annihilator is not typically singly generated, so we will often have to work with ideals rather than minimal functions.

Now that we have set up terminology and notation, we may describe the content and organization of the paper. Section \ref{S:Prelims} deals with preliminaries on multivariate spectral theory, reproducing kernels and $K$-contractions. Several of the theorems therein are based on results known in other contexts  but we chose to provide complete arguments for completeness. Most notably, Theorems \ref{T:LatIsomThm} and \ref{T:QCIsom} extend results mentioned in \cite[Section 2]{davidsonramseyshalit2015} for the Drury--Arveson space. Our novel contributions are found in subsequent sections.

In Section \ref{S:Spec}, we introduce one of the main actors of the paper, namely the so-called \emph{support} of an ideal of multipliers. We view this set as analogous to the support of an inner function, as defined previously. The definition that we give of the support is not as transparent as the explicit function theoretic description that is available in $H^\infty(\bD)$, however. We remedy this shortcoming somewhat by relating the support to what we call the \emph{approximate zero set} of the ideal (see Theorems \ref{T:suppEqAZ} and \ref{T:suppAZTheta}). One of our main results is the following (Theorem \ref{T:suppEqSpec}), which shows that for the concrete functional models, the support precisely encodes spectral information, and a perfect analogue of \eqref{Eq:SpecEqSuppd1} holds.

\begin{theorem}
	Let $\frk{a}$ be a weak-$*$ closed ideal of $\MH$.
	Denote by $\kZa$ the compression of the coordinate multipliers to the co-invariant subspace $\mc{H}\ominus[\frk{a}\mc{H}]$. Let $\B_\fa$ denote the weak-$*$ closed unital algebra generated by $\kZa$. Then, $\supp(\fa)$ coincides with both the Taylor spectrum of $\kZa$ and the Banach algebraic spectrum of $\kZa$ in $\B_\fa$.
	\label{T:Intro:Spec}
\end{theorem}

In Section \ref{S:Loc}, we examine whether \eqref{Eq:SpecEqSuppd1} is valid for general absolutely continuous $K$-contractions. We observe in 
Lemma \ref{L:wSpecInSupp} and Example \ref{E:SmallSpecLargeSupp} that while one inclusion always holds, it may be strict in general. We introduce a certain notion of \emph{localizability} of an ideal, which aims to fix this pathology. The following is Theorem \ref{T:LocInSigPrime}. 

\begin{theorem}
	Let $T$ be an absolutely continuous $K$-contraction and let $\sigma(T)$ denote its Banach algebraic spectrum in the weak-$*$ closed algebra that $T$ generates. Let $\fa$ denote the annihilator of $T$. Then, we have that
	\[ \Loc(\fa)\subset \sigma(T) \subset \supp(\fa). \]
	\label{T:Intro:LocLower}
\end{theorem}

Notably, the localizable set of a weak-$*$ closed ideal in $H^\infty(\bD)$ coincides with its support, so that the previous result is a genuine generalization of \eqref{Eq:SpecEqSuppd1}. On the other hand, the utility of the localizable set as a lower bound on $\sigma(T)$ is highly dependent on the dimension of the zero set of the annihilator. For instance, we show in Theorem \ref{T:PrimeIdealEmptyLoc} that the localizable set of a prime ideal is empty  when its zero set has positive dimension as a complex analytic variety. Conversely, when the zero set is zero-dimensional or the support is highly disconnected, more information can be extracted (see Corollary  \ref{C:LocDimZero} and Theorem \ref{T:TotDiscon}). In the following, we assume that the sequence $(b_n)_n$ associated to $K$ satisfies $\sum_{n=1}^\infty b_n=1$.

\begin{theorem}
	Let $T$ be an absolutely continuous $K$-contraction and let $\sigma(T)$ denote its Banach algebraic spectrum in the weak-$*$ closed algebra that it generates. Let $\fa$ denote the annihilator of $T$. Then, the following statements hold.
	\begin{enumerate}[{\rm (i) }]
	 \item If the zero set of $\fa$ is discrete, then
	\[ \Loc(\fa)\cap\BB_d=\sigma(T)\cap\BB_d=\supp(\fa)\cap\BB_d. \]
	\item If the support of $\fa$ is totally disconnected, then
	\[ \Loc(\fa)=\sigma(T)=\supp(\fa). \]
	\end{enumerate}
	\label{T:Intro:Zdim0}
\end{theorem}

Similar results are obtained for the Taylor spectrum, under a natural purity condition on the $K$-contraction (see Theorem \ref{T:PureDimZero}).

It should be noted that in the classical situation of $H^\infty(\bD)$, the part of the support of an ideal that lies in $\DD$ is always discrete. On the other hand,  the support as a whole need not be totally disconnected.  It is thus conceivable that part (ii) of Theorem \ref{T:Intro:Zdim0} holds under an assumption weaker than total disconnectedness of the support. Indeed, much is still to be understood on this topic, especially in regards to the spectral points that lie in the unit sphere.

\textbf{Acknowledgements.} We would like to express our thanks to J\"{o}rg Eschmeier for his advice on matters of spectral theory.
His suggestions have allowed for a considerable simplification of Section \ref{S:Prelims}.
We would also like to thank Michael Hartz for pointing us to \cite[Section 3]{richter2010}.

\section{Preliminaries}\label{S:Prelims} 

\subsection{Multivariate spectra}\label{SS:multispectra} 
Let $\fH$ be a (complex) Hilbert space $\fH$. We denote by $\B(\fH)$ the $\rC^*$-algebra of bounded linear operators on $\fH$. 
If $\S\subset \B(\fH)$ is a subset, then we denote its commutant by
\[
 \S'=\{A\in \B(\fH):AX=XA \text{ for every } X\in \S\}.
\]
Let $\A\subset \B(\fH)$ be a unital norm-closed subalgebra.
Let $T_1,\ldots,T_d$ be operators in $\A'\cap\A$ and set $T=(T_1,\cdots,T_d)$.
We define $\sigma_\A(T)$ to be the subset of $\bC^d$ consisting of those points $z=(z_1,\ldots,z_d)$ such that $I\nin{\textstyle\sum_{j=1}^d}(T_j-z_jI)\mc{A}$.
When $\mc{A}$ is commutative, it is well known that
\[
	\sigma_{\mc{A}}(T)=\{(\chi(T_1),\cdots,\chi(T_d)):\chi\in\Char\mc{A}\}
\]
where $\Char(\mc{A})$ denotes the set of all unital multiplicative linear functionals on $\mc{A}$.
In this paper, we will mainly be interested in the following two special cases of this construction.
When $\A$ is the weak-$*$ closed unital subalgebra of $\B(\fH)$ generated by $T_1,\ldots,T_d$, then we write
\[
 \sigma_\A(T)=\wSpec(T)
\]
and call this the \textit{weak-$*$ spectrum} of $T$. When $\A=\{T_1,\ldots,T_d\}'$, then we write
\[
 \sigma_\A(T)=\sigma'(T)
\]
and call this the \textit{commutant spectrum} of $T$. It readily follows from these definitions that
\[
	\sigma'(T)\subset\wSpec(T).
\]
One useful property of the commutant spectrum is the following. Let $\fM\subset \fH$ be a closed subspace which is invariant for $\{T_1,\ldots,T_d\}'$. Put $T|_\fM=(T_1|_\fM, \ldots, T_d|_\fM)$. Then, it is easy to verify that
\begin{equation}\label{Eq:ComSpecRestrict}
 \sigma'(T|_{\frk{M}})\subset\sigma'(T).
 \end{equation}
Besides the commutant and weak-$*$ spectra, there is a third type of joint spectrum associated a commuting $d$-tuple $T$ that we study in this paper: the \emph{Taylor joint spectrum} \cite{Taylor1970Spec}. We denote it by $\TSpec(T)$ and refer the reader to \cite[Section 25]{muller2007} for the precise definition. It follows from \cite[Proposition 25.3 and Theorem 25.4]{muller2007} that $\TSpec(T)$ is a compact subset of $\bC^d$. Moreover, the Taylor spectrum is empty precisely when $\H$ is the zero space (\cite[Proposition 7.6]{muller2007}). Invoking \cite[Proposition 25.3]{muller2007}, we see that
\begin{equation}\label{Eq:incTayComm}
 \TSpec(T)\subset \sigma'(T).
\end{equation}
If we put $T^*=(T_1^*,\ldots,T_d^*)$, then the discussion in \cite[Section 3]{richter2010} shows that
\begin{equation}\label{Eq:TayAdj}
\TSpec(T^*)=\{(\cc{z_1},\ldots,\ol{z_d}): (z_1,\ldots,z_d)\in\TSpec(T)\}. 
\end{equation}
Next, we define the \textit{joint point spectrum} $\pSpec(T)$ to be the subset of $\bC^d$ consisting of those points $z=(z_1,\ldots,z_d)$ such that there is a non-zero vector $h\in \H$ with $T_j h=z_j h$ for every $1\leq j\leq d$. It is a consequence of \cite[Remark 25.2]{muller2007} that
\begin{equation}\label{Eq:incTayPoint}
 \pSpec(T)\subset \TSpec(T).
\end{equation}
The main advantage of the Taylor spectrum is that it has a rich functional calculus \cite{Taylor1970FC}, which we describe as follows. 

Let $x_1,\dots,x_d$ denote the canonical coordinate functions for $\CC^d$ and its subsets.  Let $U\subset\CC^d$ be an open set containing $\TSpec(T)$. We denote by $\mc{O}(U)$ the algebra of analytic functions on $U$, endowed with the topology of uniform convergence on compact subsets of $U$. 
Then, for every relatively compact  open subset $V$ such that $\TSpec(T)\subset V\subset \ol{V}\subset U$  there is a constant $C_{T,V}>0$ and a unital algebra homomorphism 
	\[
	\tau_{T,U}:\mc{O}(U)\to  \{T_1,\ldots,T_d\}''
	\]
	such that 
	\[
	\tau_{T,U}(x_j)=T_j, \quad 1\leq j\leq d
	\]
	and
	\[
	 \|\tau_{T,U}(f)\|\leq C_{T,V}\sup_{z\in V}|f(z)|
	 \]
	 for every $f\in \O(U)$ (see \cite[Theorem 5.18]{curto1988}). The following summarizes the additional properties of the Taylor functional calculus that we will need.	 
	 
	\begin{theorem}\label{T:Taylorprop}
	Let $T=(T_1,\ldots,T_d)$ be a commuting $d$-tuple of operators on some Hilbert space $\fH$. Let $U\subset \bC^d$ be an open set containing $\TSpec(T)$. Let 
	\[
	\tau_{T,U}:\O(U)\to B(\fH)
	\]
	be the Taylor functional calculus. Then, the following statements hold.
	\begin{enumerate}[{\rm (i)}]
		\item If $V\subset \bC^d$ is another open set containing $\TSpec(T)$, and if $f\in\mc{O}(U)$ and $g\in\mc{O}(V)$ are functions such that $f|_{U\cap V}=g|_{U\cap V}$, then
			\[ \tau_{T,U}(f)=\tau_{T,V}(g). \]
		\item  Let $R=(R_1,\dots,R_d)$ be a commuting $d$-tuple of operators on a Hilbert space $\frk{K}$ with $\TSpec(R)\subset U$.
			Let $X\in\mc{B}(\frk{H},\frk{K})$ be such that $XT_j=R_jX$ for $j=1,\dots,d$.
			Then
			\[ X\tau_{T,U}(f)=\tau_{R,U}(f)X \]
			for each $f\in \mc{O}(U)$.
		\item   Let $K_1$ and $K_2$ be disjoint non-empty compact subsets of $\CC^d$ such that $\TSpec(T)=K_1\cup K_2$, and suppose $U_1$ and $U_2$ are open disjoint neighbourhoods of $K_1$ and $K_2$, respectively. 
			Let $\chi$ denote the characteristic function of the set $U_1$, and set
			$P=\tau_{T,U_1\cup U_2} (\chi). $
			Then, $P$ is a non-zero idempotent operator commuting with $T$ which satisfies
			\[ \TSpec(T|_{\ran P})=K_1 \qand \TSpec(T|_{\ran(I-P)})=K_2. \]
    \item Let $\fE$ be a Hilbert space. Then \[ \TSpec(T\otimes I_{\frk{E}})=\TSpec(T). \]
    
    \item For each $f\in \O(U)$, we have that
\[
 \tau_{T\otimes I_\fE,U}(f)=\tau_{T,U}(f)\otimes I_\fE.
\]
	\end{enumerate}
\end{theorem}
\begin{proof}
	Statements (i),(ii) and (iii) can be found in Theorem III.13.5 along with Corollaries III.9.10 and III.9.11 of \cite{VasilescuBook}. Furthermore, if we let $n=\dim \fE$, then because $T\otimes I_\fE$ is unitarily equivalent to the $n$-fold direct sum of $T$, (iv) readily follows from the definition of the Taylor spectrum.

	We turn to (v). We consider the unital continuous homomorphism $\Phi:\O(U)\to B(\fH\otimes \fE)$ defined as
	\[
	\Phi(f)= \tau_{T,U}(f)\otimes I_\fE, \quad f\in \O(U).
	\]
 	 Note that by (iv) the set $U$ is an open neighbourhood of $\TSpec(T\otimes I_\fE)$. 
Next, let $g_1,\ldots,g_m\in\mc{O}(U)$ and observe that
	\begin{align*}
	 (g_1,\ldots,g_m)(\TSpec(T\otimes I_{\fE}))&=(g_1,\ldots,g_m)(\TSpec(T))\\
	 &=\TSpec(\tau_{T,U}(g_1),\ldots,\tau_{T,U}(g_m))
	\end{align*}
	by the spectral mapping property of the Taylor spectrum \cite[Corollary 30.11]{muller2007}.
	Applying (iv) again, we find
	\[
	 \TSpec(\tau_{T,U}(g_1),\ldots,\tau_{T,U}(g_m))=\TSpec(\tau_{T,U}(g_1)\otimes I_\fE,\ldots,\tau_{T,U}(g_m)\otimes I_\fE)
	\]
    whence 
    \[
      (g_1,\ldots,g_m)(\TSpec(T\otimes I_{\fE}))=\TSpec(\tau_{T,U}(g_1)\otimes I_\fE,\ldots,\tau_{T,U}(g_m)\otimes I_\fE).
    \]
	By the uniqueness of the Taylor functional calculus \cite[Theorem 5.20]{curto1988}, we conclude that $\tau_{T\otimes I_\fE,U}=\Phi$.
\end{proof}

\subsection{Regular unitarily invariant kernels and spaces}\label{SS:RKHS} 
Let $d\geq 1$ be a fixed positive integer and let $\bB_d\subset \bC^d$ denote the open unit ball. Hereafter, $\mc{H}$ is a reproducing kernel Hilbert space of functions on $\BB_d$.
Thus, for every $z\in\BB_d$, there exists a $k_z\in\mc{H}$ such that $\ip{h,k_z}=h(z)$ for every $h\in\mc{H}$. The functions $\{k_z:z\in \bB_d\}$ span a dense subset of $\H$.
The \textit{kernel} of $\mc{H}$ is the function $K:\bB_d\times \bB_d\to \bC$ defined as $K(w,z)=\ip{k_z,k_w}$. Since the space $\H$ and the kernel $K$ uniquely determine one another, we will often speak of properties of the space or of the kernel interchangeably.

Throughout, we assume the following.
\begin{enumerate}
	\item[(i)] $K$ is \textit{normalized at $0$}; that is, $K(z,0)=1$ for every $z\in\BB_d$.
	\item[(ii)] $K$ is \textit{irreducible}; that is, $K(z,w)\neq 0$ for every $z,w\in\BB_d$, and $k_z$ and $k_w$ are linearly independent when $z\neq w$.
	\item[(iii)] $K$ has the \textit{complete Nevanlinna--Pick property}; that is, the function
		\[ (z,w)\mapsto 1-1/K(z,w) \]
		is positive semi-definite.
	\item[(iv)] $K$ is analytic in the first variable.
\end{enumerate}
Note that (i) implies that $1=k_0\in\mc{H}$, and it follows from (ii) that $\|k_z\|\neq 0$ for every $z\in\BB_d$.
As explained in \cite[Section 7.1]{agler2002}, property (iii) is equivalent to the usual matrix interpolation condition.
By (iv), we see that $k_z\in\mc{O}(\BB_d)$ for each $z$ whence $\mc{H}\subset \mc{O}(\BB_d)$.

The \textit{multiplier algebra of $\mc{H}$}, denoted by $\mc{M}(\mc{H})$, consists of all those functions $f$ on $\BB_d$ such that $fh\in\mc{H}$ whenever $h\in\mc{H}$.
Given $f\in\mc{M}(\mc{H})$, we set $M_fh=fh$ for each $h\in\mc{H}$.
An application of the closed graph theorem demonstrates that $M_f\in\mc{B}(\mc{H})$ for every $f\in\mc{M}(\mc{H})$. We may thus define a norm on $\M(\H)$ as
\[
 \|f\|_{\M(\H)}=\|M_f\|_{\B(\H)}, \quad f\in \M(\H).
\]
It is readily seen that $M_f^*k_z=\cc{f(z)}k_z$ for every $f\in \M(\H)$ and every $z\in \bB_d$, whence
\[ \sup_{z\in\BB_d}|f(z)|\leq \|f\|_{\M(\H)}. \]
In particular, we see that evaluation at a point in $\bB_d$ gives rise to a character on $\M(\H)$. We say that $\H$ satisfies the \emph{Corona property} if the characters of evaluation at some point in $\bB_d$ are weak-$*$ dense in $\Char(\MH)$.

Via the identification $f\mapsto M_f$ we may view $\M(\H)$ as a unital subalgebra of $\B(\H)$. As explained in \cite[Remark 5.2]{hartz2017isom}, $\M(\H)$ is always dense in $\H$ in our context, so an elementary calculation based on \cite[Theorem 2.35]{agler2002} shows that 
\begin{equation}\label{Eq:commMH}
 \M(\H)=\M(\H)'=\M(\H)''.
\end{equation}
In particular, $\M(\H)$ is a weak-$*$ closed subalgebra of $\B(\H)$.

We will also require vector-valued versions of multipliers. If $\fE$ is a Hilbert space, then using standard identifications we may view $\H\otimes \fE$ has a Hilbert space of $\fE$-valued functions on $\bB_d$. Given another Hilbert space $\fF$, we denote by $\mc{M}(\mc{H}\otimes\frk{E},\mc{H}\otimes\frk{F})$ the linear space consisting of all those functions $\Phi:\BB_d\to\mc{B}(\frk{E},\frk{F})$ with the property that $\Phi h\in\mc{H}\otimes\frk{F}$ for every $h\in\mc{H}\otimes\frk{E}$. If $\Theta\in \mc{M}(\mc{H}\otimes\frk{E},\mc{H}\otimes\frk{F})$, then we may define a bounded linear map $M_\Theta:\mc{H}\otimes\frk{E}\to \mc{H}\otimes\frk{F}$ by putting
\[
 (M_\Theta h)(z)=\Theta(z)h(z)
\]
for $z\in\BB_d$ and $h\in \mc{H}\otimes\frk{E}$.
We say that $\Theta$ is an \textit{inner} multiplier if $M_{\Theta}$ is a partial isometry.

 Our next order of business is to exhibit a certain duality between weak-$*$ closed ideals of multipliers and invariant subspaces. Such a duality is well known in the context of the Drury--Arveson space \cite[Theorem 2.4]{davidsonramseyshalit2015}, but we lack a reference that is applicable in the generality that we require. Thus, we provide an argument. 
 
Let $\N\subset \H$ be a closed $\M(\H)$-invariant subspace. We associate to it a weak-$*$ closed ideal of $\M(\H)$ defined as
\[
 \iota(\N)=\{f\in\MH: \ran M_f \subset \N\}.
\]
Furthermore, given a weak-$*$ closed ideal $\frk{a}$ of $\MH$ and a subset $X\subset \H$, we use the notation
\[
[\fa X]=\ol{\spn\{f x:f\in \fa,x\in X\}}.
\]
We also set $\rho(\fa)=[\fa \H]$, which is seen to be a closed $\M(\H)$-invariant subspace. We now show that the maps $\rho$ and $\iota$ are  mutual inverses by adapting the proof of \cite[Theorem 2.1]{davidson1998alg}.

\begin{theorem}\label{T:LatIsomThm}
	Let $\frk{a}$ be a weak-$*$ closed ideal of $\MH$ and $\N$ a closed $\MH$-invariant subspace of $\mc{H}$.
	Then
	\[ \rho(\iota(\N))=\N, \quad \iota(\rho(\frk{a}))=\frk{a}. \]
\end{theorem}
	\begin{proof}
	
	It is readily verified that
\[ \frk{a}\subset \iota(\rho(\frk{a})), \quad \rho(\iota(\N))\subset\N, \]
and that the maps $\iota$ and $\rho$ are order preserving.
		
We first show that $\N\subset \rho(\iota(\N))$. To see this, apply \cite[Theorem 0.7]{mccullough2000} to find a separable Hilbert space $\fE$ and an inner multiplier $\Theta\in\mc{M}(\mc{H}\otimes\frk{E},\mc{H})$ such that $\N=\ran M_\Theta$. We let $F\in \H\otimes \fE$ and we must show that $\Theta F\in \rho(\iota(\N))$. For this purpose, let $\{e_n\}_{n\in J}$ be an orthonormal basis for $\frk{E}$ and for each $n\in J$ define $\grG_n:\bB_d\to B(\bC,\fE)$ to be the constant function $e_n$. Clearly, we have that $\Gamma_n\in \M(\H,\H\otimes \fE)$ for every $n$.
Given $F\in\mc{H}\otimes\frk{E}$, there are functions $f_n\in \H$ such that $F=\sum_{n\in J} f_n \otimes e_n$. Hence, we see that
		$
		F=\sum_{n\in J} \Gamma_n f_n
		$
		so that $\Theta F=\sum_{n\in J} \Theta \Gamma_n f_n$. 
		By definition, we have $\Theta\Gamma_n\in \iota(\N)$, whence $\Theta F\in \rho(\iota(\N))$ as desired.

		We now turn to establishing $\iota(\rho(\frk{a}))\subset \fa$. For this purpose, we first claim that 
		\[
		 [\iota(\rho(\frk{a}))h] =  [\frk{a}h]
		\]
		for every $h\in \H$. Indeed, fix $h\in \H$. Using that $\fa$ is an ideal we see that 
		\[[\frk{a}h]=[\frk{a}\MH h]=[\frk{a}[\MH h]]. 
		\]
Again, we apply \cite[Theorem 0.7]{mccullough2000} to find a separable Hilbert space $\fF_h$ and an inner multiplier $\Phi_h\in\mc{M}(\mc{H}\otimes\frk{F}_h,\mc{H})$ such that $[\MH h]=\ran M_{\Phi_h}$.
		  Note now that if $\psi\in \M(\H)$, then we have
        \[
         M_\psi M_{\Phi_h}=M_{\Phi_h} (M_\psi\otimes I_{\fF_h}).
        \]
        This implies that
        \[ [\fa h]=[\frk{a}[\MH h]]=[\frk{a}\ran M_{\Phi_h}]=M_{\Phi_h}[\rho(\frk{a})\otimes\frk{F}_h] \]
		and likewise
		\[ [\iota(\rho(\frk{a}))h]=[\iota(\rho(\frk{a}))[\MH h]] = M_{\Phi_h}[\rho(\iota(\rho(\frk{a})))\otimes\frk{F}_h]. \]
		By the previous paragraph, we have that $\rho\of\iota\of\rho=\rho$ whence the previous two equations combine to yield
		\[ [\iota(\rho(\frk{a}))h] = M_{\Phi_h}[\rho(\frk{a})\otimes\frk{F}_h] = [\frk{a}h] \]
        and thus the claim is established. 
        
        Going back to proving the inclusion $\iota(\rho(\frk{a}))\subset \fa$, we suppose there exists a function $\psi\in\iota(\rho(\frk{a}))$ which does not lie in $\frk{a}$.
		A standard consequence of the Hahn--Banach theorem gives the existence of a weak-$*$ continuous linear functional $\lambda:\MH\to\CC$ such that $\lambda(\frk{a})=\{0\}$ and $\lambda(\psi)\neq 0$.
		Using that $K$ has the complete Nevanlinna--Pick property, we may now invoke \cite[Corollary 5.3]{davidsonhamilton2011} to see that there exist $h_1,h_2\in\mc{H}$ such that $\lambda(\theta)=\ip{M_\theta h_1,h_2}$ for each $\theta\in\MH$.
		Because $\lambda$ annihilates $\frk{a}$, we infer that have 
		$h_2\in\mc{H}\ominus [\frk{a}h_1]$. In turn, by the previous paragraph, we conclude that $h_2\in\mc{H}\ominus [\iota(\rho(\frk{a}))h_1]$ which forces $\lambda(\psi)=0$ and gives a contradiction.
	\end{proof}

If $\fa\subset \M(\H)$ is a weak-$*$ closed ideal, we put $\H(\fa)=\H\ominus [\fa \H]$. Furthermore, we let $\B_\fa\subset \B(\H(\fa))$ denote the algebra consisting of operators of the form
\[
 P_{\mc{H}(\frk{a})}M_{f}|_{\mc{H}(\frk{a})} , \quad f\in \M(\H).
\]
Using Theorem \ref{T:LatIsomThm} and a standard argument, we can identify $\B_\fa$ as a quotient of $\M(\H)$.   
	
	\begin{theorem}\label{T:QCIsom} 
	The algebra $\B_\fa$ is weak-$*$ closed. Moreover, the map $\grG:\MH/\frk{a}\to\mc{B}_{\frk{a}}$ given by 
		\[ \grG(f+\frk{a})=P_{\mc{H}(\frk{a})}M_{f}|_{\mc{H}(\frk{a})},\quad f\in \M(\H) \]
		is a unital completely isometric isomorphism and weak-$*$ homeomorphism.
	\end{theorem}
	\begin{proof}
		First note that it follows from the commutant lifting theorem \cite[Theorem 5.1]{BTV2001} that $\mc{B}_{\frk{a}}=\mc{B}_{\frk{a}}'$, and thus that $\mc{B}_{\frk{a}}$ is closed in the weak-* topology.
		Define $\grG_0:\MH\to\mc{B}_{\frk{a}}$ by setting $\grG_0(f)=P_{\mc{H}(\frk{a})}M_f |_{\mc{H}(\frk{a})}$ for each $f\in\MH$.
		From its definition, we readily see that $\grG_0$ is unital, completely contractive, surjective, and weak-$*$ continuous.
		Because $\mc{H}(\frk{a})$ is $\MH$-coinvariant, $\grG_0$ is also a homomorphism.
		Note that
		\begin{eqnarray*}
			\ker\grG_0 &=& \{f\in\MH: P_{\mc{H}(\frk{a})}M_f=0 \} \\
			&=& \{f\in\MH: \ran M_f\subset \rho(\frk{a}) \} \\
			&=& \iota(\rho(\frk{a}))
		\end{eqnarray*}
		whence $\ker\grG_0=\frk{a}$ by  Theorem \ref{T:LatIsomThm}.
		Thus $\grG_0$ induces a unital completely contractive isomorphism $\grG:\MH/\frk{a}\to\mc{B}_{\frk{a}}$ that is weak-$*$ continuous. We now claim that $\Gamma$ is completely isometric. To see this, let $n$ be a positive integer and let $[f_{ij}]$ be an $n\times n$ matrix with entries in $\M(\H)$. 
		Another application of the commutant lifting theorem yields an $n\times n$ matrix of multipliers $[g_{ij}]$ with $\|[\Gamma(f_{ij})]\|=\|[M_{g_{ij}}]\|$ such that $[\Gamma(f_{ij})]=[\Gamma(g_{ij})]$. In particular, this means that $f_{ij}-g_{ij}\in \fa$ for every $i,j$ and
		\[
		 \|[\Gamma(f_{ij})]\|=\|[M_{g_{ij}}]\|\geq \|[{f_{ij}}+\fa] \|
		\]
which shows that $\Gamma$ is completely isometric. Finally, by \cite[Theorem A.2.5(3)]{BlecherLeMerdy2004} we conclude $\grG$ is a weak-$*$ homeomorphism.
	\end{proof}

    Throughout the paper, we will be interested in kernels that exhibit additional structure. 	We say that the kernel $K$ is \textit{unitarily invariant} if 
	\[K(Uz,Uw)=K(z,w)\] for every $z,w\in\BB_d$ and unitary operator $U$ on $\CC^d$. As shown in \cite[Lemma 2.2]{hartz2017isom}, this is equivalent to the existence of a unique sequence $(a_n)_n$ of non-negative numbers such that
	\[ K(z,w)=1+\sum_{n=1}^\infty a_n\ip{z,w}^n \]
	for every $z,w\in\BB_d$. We say that the unitarily invariant kernel $K$ is \textit{regular} if the associated sequence $(a_n)_n$ is strictly positive and satisfies
\[ \lim_{n\to\infty}\frac{a_n}{a_{n+1}}=1. \]
 By \cite[Lemma 2.3]{hartz2017isom} and \cite[Lemma 2.1]{CH2016}, there exists a non-negative sequence $(b_n)$ such that
\begin{equation}\label{Eq:bSeries}
	1-\frac{1}{K(z,w)}=\sum_{n=1}^\infty b_n\ip{z,w}^n
\end{equation}
for $z,w\in \BB_d$. It is readily seen that $b_1=a_1$ and $\sum_{n=1}^\infty b_n\leq 1$. We say that the space $\H$ is \emph{maximal} when $\sum_{n=1}^\infty b_n=1$. This is equivalent to the fact that the series $\sum_{n=1}^\infty a_n$ diverges.

Such spaces are plentiful; we exhibit next an entire scale of examples.

\begin{example}\label{E:Hs}
 Let $s$ be a non-positive real number and let $\H_s$ be the reproducing kernel Hilbert space on $\bB_d$ with kernel
 \[
  K_s(z,w)=1+\sum_{n=1}^\infty (n+1)^{-s} \langle z,w \rangle^n, \quad z,w\in \bB_d.
 \]
 The fact that $K_s$ has the complete Nevanlinna--Pick property may be shown as in \cite[Corollary 7.41]{agler2002}. We note that $\H_s$ is a regular unitarily invariant space which is maximal. In particular, by choosing $s=0$ we recover the Drury--Arveson space. When $d=1$ and $s=-1$, $\H_s$ is the Dirichlet space $\D$ on the disc. 

Next, let $0< \sigma \leq 1/2$ and consider the kernel
\[
 B_\sigma(z,w)=\frac{1}{(1-\langle z,w\rangle )^{2\sigma}}, \quad z,w\in \bB_d.
\]
The corresponding reproducing kernel Hilbert space on $\bB_d$ satisfies the Corona property by \cite[Corollary 3]{costea2011}. On the other hand, it is known \cite[Example 7.2]{hartz2017isom} that this space is isomorphic to $\H_{2\sigma -1}$. This shows that for $-1<s\leq 0$ the space $\H_s$ is regular, maximal, unitarily invariant, and it has the Corona property. 
 \qed
\end{example}

The following result is a consequence of \cite[Proposition 8.5]{hartz2017isom}
	\begin{theorem}\label{T:GleasonTrick}
		Let $\H$ be a regular unitarily invariant space. 
		Let $z\in\BB_d$ and $f\in\MH$. Then there exist $f_1,\ldots,f_d\in\MH$ such that
		\[ f=f(z)+\sum_{j=1}^d(x_j-z_j)f_j. \]
	\end{theorem}

	Henceforth, we take $\bN=\{0,1,2,\ldots\}$; we choose to include zero for convenience.
Given $\alpha=(\alpha_1,\ldots,\alpha_d)\in\NN^d$, we set 
\[ 
 |\alpha|=\sum_{j=1}^d \alpha_j, \quad \alpha! = \prod_{j=1}^d \alpha_j!\]
and if $r=(r_1,\ldots,r_d)$ is a $d$-tuple of commuting elements in a unital algebra, then we set
\[ r^\alpha = r_1^{\alpha_1}\cdots r_d^{\alpha_d}. \]
		When $\H$ is a regular unitarily invariant space, the set $\{x^\alpha: \alpha\in\NN^d\}$ is an orthogonal basis for $\mc{H}$ with $\|1\|=1$ and
	\[ \|x^\alpha\|^2_{\mc{H}}=\frac{1}{a_{|\alpha|}}\frac{\alpha!}{|\alpha|!} \]
	for each $\alpha\in\NN^d$ with $|\alpha|>0$ (see \cite[Section 4]{greene2002}). In particular, $\CC[x_1,\dots,x_d]$ is a dense subset of $\mc{H}$.
	Furthermore, it follows from \cite[Theorem 6.4]{hartz2017isom} that
the multiplier algebra $\M(\H)$ contains all polynomials.  We put
\[
 M_x=(M_{x_1},\ldots,M_{x_d}).
\]
If $\fa\subset \M(\H)$ is a weak-$*$ closed ideal, then we put
\[\kZa=P_{\mc{H}(\frk{a})}M_x|_{\mc{H}(\frk{a})}. \]
For $f\in \M(\H)$, we write
\[
 f(\kZa)=P_{\mc{H}(\frk{a})}M_f|_{\mc{H}(\frk{a})}.
\]
In particular, we see then that
\[
 \B_\fa=\{f(\kZa):f\in \M(\H)\}
\]
and this algebra can be identified with $\M(\H)/\fa$ by Theorem \ref{T:QCIsom}.

We denote by $\AH$ the norm closure of $\CC[x_1,\ldots,x_d]$ inside of $\M(\H)$. Because the multiplier norm always dominates the supremum norm over $\bB_d$, it is readily verified that multipliers in $\AH$ extend to  continuous functions on the closed unit ball.

\subsection{$K$-Contractions and functional models} 

Let $\H$ be a regular unitarily invariant space on $\bB_d$ with kernel $K$. Assume that
\[
 1-\frac{1}{K(z,w)}=\sum_{n=1}^\infty b_n \langle z,w \rangle^n, \quad z,w\in \bB_d.
\]
We say that a $d$-tuple $T=(T_1,\ldots,T_d)$ of commuting operators on some Hilbert space $\fH$ is a \textit{$K$-contraction} if
\[ \sum_{n=1}^\infty b_n\sum_{|\alpha|=n}\frac{n!}{\alpha!} T^\alpha T^{*\alpha} \leq I. \]
Recall that commuting $d$-tuple $U=(U_1,\ldots,U_d)$ is said to be a \emph{spherical unitary} if $U_1,\ldots,U_d$ are all normal and $\sum_{j=1}^d U_j U_j^*=I$. Note that 
\[ \sum_{n=1}^\infty b_n\sum_{|\alpha|=n}\frac{n!}{\alpha!} U^\alpha U^{*\alpha} =\sum_{n=1}^\infty b_n\left(\sum_{j=1}^d U_j U_j^* \right)^n=\sum_{n=1}^\infty b_n I \leq I \]
so that $U$ is a $K$-contraction.

We gather some useful known facts regarding commuting $K$-contractions.

\begin{theorem}
    Let $T=(T_1,\ldots,T_d)$ be a commuting $d$-tuple of operators on some Hilbert space $\fH$. Consider the following statements. 
    \begin{enumerate}[{\rm (i)}]
    \item The $d$-tuple $T$ is a commuting $K$-contraction.
    \item There exist two Hilbert spaces $\fE$ and $\frk{K}$, a spherical unitary $U$ on $\frk{K}$, and an isometry $V:\fH\to(\H\otimes \fE)\oplus \fK $ such that 
     \[
     VT^*=((M_x^*\otimes I_\fE)\oplus U^*)V.
     \]

     \item We have that $\wSpec(T)\subset\cc{\BB}_d$.
      \item There exists a unique unital completely contractive homomorphism $\alpha_T:\AH\to\mc{B}(\frk{H})$ satisfying $\alpha_T(x_j)=T_j$ for $j=1,\ldots,d$.
    \end{enumerate}
Then, we have that
\[
 {\rm (i)} \Leftrightarrow {\rm (ii)} \Rightarrow {\rm (iii)}, {\rm (iv)}.
\]
\begin{proof}
 This follows from Lemma 5.3 and its proof, along with Theorems 5.4 and 5.6 in \cite{CH2016}.
\end{proof}
	\label{T:KCntrcFacts}
\end{theorem}

In view of (ii) above, we usually refer to $M_x$ as being the \emph{functional model} for $K$-contractions. Given a commuting $K$-contraction $T$, we refer to the map $\alpha_T$ in (iv) above as the \textit{$\AH$-functional calculus}. When $\alpha_T$ extends to a weak-$*$ continuous unital homomorphism $\MHFC_T: \MH\to\mc{B}(\frk{H})$, we say that $T$ is an \textit{absolutely continuous $K$-contraction} and call $\widehat{\gra}_T$ the \textit{$\MH$-functional calculus for $T$}. This choice of terminology is justified by some known characterizations of absolute continuity \cite{CD2016abscont},\cite{BHM2018}, but we will not require those results for our purposes in this paper. When $T$ is an absolutely continuous $K$-contraction, we define its \textit{annihilator} to be the weak-$*$ closed ideal
\[ \Ann(T)=\{f\in \MH: \MHFC_T(f)=0\}. \]
It is easy to see that $T$ is absolutely continuous whenever the spherical unitary part is absent from the dilation in part (ii) of Theorem \ref{T:KCntrcFacts}. When a commuting $K$-contraction $T$ has this property, we will say that it is \emph{$K$-pure}. 
We now refine the dilation result found in Theorem \ref{T:KCntrcFacts}.

\begin{corollary}\label{C:GetInter}
    Let $\H$ be a regular unitarily invariant space with kernel $K$. Let $T=(T_1,\ldots,T_d)$ be a $K$-pure commuting $K$-contraction on some Hilbert space $\fH$. Put $\fa=\Ann(T)$.
    Then, there exist a Hilbert space $\fE$ and an isometry $V:\fH\to\H(\fa)\otimes \fE$ such that 
    \[
     VT^*=((\kZa)^*\otimes I_\fE)V.
     \]
	Moreover, $\ker V^*$ contains no non-zero closed subspace of the form $\frk{M}\otimes\frk{E}$ where $\fM\subset \H(\fa)$ is invariant for $\kZa$ . 
\end{corollary}
\begin{proof}
    By assumption, there exist a Hilbert space $\fE$ and an isometry $W:\fH\to\H\otimes \fE$ such that 
    \[
     WT^*=(M^*_x\otimes I_\fE)W.
     \]
    Thus, we find
    \[
     f(T)W^*=W^* (M_f\otimes I_\fE), \quad f\in \M(\H)
    \]
    and consequently
    \[
     W^* (M_f\otimes I_\fE)=0, \quad f\in \fa.
    \]
    We infer that $[\fa \H]\otimes \fE\subset \ker W^*$ and so 
    $
     \ran W\subset \H(\fa)\otimes \fE.
    $
    In particular, 
    \[
     (M^*_x\otimes I_\fE)|_{\ran W}=(M^*_x|_{\H(\fa)}\otimes I_\fE)|_{\ran W}=((\kZa)^*\otimes I_\fE)|_{\ran W}.
    \]
The operator $V=P_{\H(\fa)\otimes \fE}W:\fH\to \H(\fa)\otimes \fE $ is an isometry  which satisfies 
    \[
     VT^*=((\kZa)^*\otimes I_\fE)V.
     \]

Finally, let $\frk{M}$ be a closed $Z^{(\frk{a})}$-invariant subspace of $\mc{H}(\frk{a})$ such that $\frk{M}\otimes\frk{E}\subset \ker V^*$.
	Applying Theorem \ref{T:LatIsomThm} to the $\M(\H)$-invariant subspace $\fM+[\fa \H]$ provides a weak-$*$ closed ideal $\frk{b}\subset \mc{M}(\mc{H})$ such that $\fa\subset \frk{b}$ and $\frk{M}=[\frk{b}\mc{H}]\ominus[\frk{a}\mc{H}]$. In particular, we see that 
	\[
	 \ran (f(\kZa)\otimes I_\fE)\subset \ker V^*
	\]
    and thus
    \[
     f(T)V^*=V^*(f(\kZa)\otimes I_\fE)=0
    \]
    for every $f\in \frk{b}$. Since $V^*$ is surjective, we infer that $\frk{b}\subset \Ann(T)=\fa$, whence $\fa=\frk{b}$ and $\fM=\{0\}$.
\end{proof}

Because of the previous result, we will typically refer to $\kZa$ as the \emph{functional model} for $K$-pure commuting $K$-contractions whose annihilator is $\fa$.

Let $T=(T_1,\ldots,T_d)$ be a commuting $K$-contraction on some Hilbert space $\fH$. We define a linear map 
\[
 \Phi_T:\B(\fH)\to \B(\fH)
\]
by
\[
 \Phi_T(A)=\sum_{n=1}^\infty b_n\sum_{|\alpha|=n}\frac{n!}{\alpha!}T^\alpha A T^{*\alpha}, \quad A\in \B(\fH).
\]
It is readily seen that $\Phi_T$ is positive and $\Phi_T(I)\leq I$, and thus $\Phi_T$ is contractive. For each $m\geq 1$ and $A\in \B(\fH)$, we use the notation
\[
 \Phi_T^{\circ m}(A)=(\underbrace{\Phi_T\circ \ldots \circ \Phi_T}_{m})(A).
\]
We say that the $K$-contraction $T$ is \emph{$K$-asymptotically vanishing} if the sequence of positive contractions $(\Phi_T^{\circ m}(I))_m$ converges to $0$ in the strong operator topology of $\B(\fH)$. 
If $T$ is $K$-asymptotically vanishing and $\fM\subset\fH$ is a closed subspace which is invariant for $T$, then it is easily verified that $T|_\fM$ is $K$-asymptotically vanishing as well. Likewise, if $\fM$ is co-invariant for $T$, then $P_{\fM}T|_{\fM}$ is also $K$-asymptotically vanishing.

We note that the model $M_x$ is always $K$-asymptotically vanishing. To see this, for each $m\geq 1$ we let $Q_m\in \B(\H)$ denote the orthogonal projection  onto the closure of $\spn\{ x^\alpha:|\alpha|\geq m\}$.
	It is then readily seen that 
	$\Phi_{M_x}^{\of m}(I)\leq Q_m$ for every $m\geq 1$, and thus the sequence $(\Phi_{M_x}^{\of m}(I_\H))_m$ converges to $0$ in the strong operator topology of $B(\H)$.  
In addition, it is easy to verify that a spherical unitary is $K$-asymptotically vanishing if and only if $\H$ is not maximal. 
This implies that whenever $\H$ is maximal, a $K$-asymptotically vanishing commuting $K$-contraction cannot have a non-trivial reducing subspace on which it restricts to be a spherical unitary; in other words, $K$-asymptotically vanishing commuting $K$-contractions are completely non-unitary in that context. 

\begin{lemma}\label{L:Kcontevalue}
 Let $\H$ be a maximal regular unitarily invariant space with kernel $K$.
	Let $T=(T_1,\dots,T_d)$ be a $K$-asymptotically vanishing commuting $K$-contraction. If $w\in \bC^d$ satisfies $\|w\|=1$, then $w\notin \pSpec(T)$.
\end{lemma}
\begin{proof}

Let $\fH$ be the Hilbert space on which $T_1,\ldots,T_d$ act. Let \[\fH_w=\{x\in\frk{H}: T_jx=w_jx\text{ for }j=1,\ldots,d\}.\]
With respect to the decomposition $\fH_w\oplus \fH_w^\bot$, for each $\alpha\in \bN^d$ we have
	\[ T^\alpha=\begin{bmatrix} w^\alpha I_{\fH_w} & A_\alpha \\ 0 & S_\alpha \end{bmatrix} \]
	for some $d$-tuples $A_\alpha, S_\alpha$. A routine calculation reveals that the $(1,1)$-entry of the operator 
	\[
	 \sum_{n=1}^\infty b_n\sum_{|\alpha|=n}\frac{n!}{\alpha!} T^\alpha T^{*\alpha}
	\]
is	 \[
\sum_{n=1}^\infty b_n\sum_{|\alpha|=n}\frac{n!}{\alpha!}(w^\alpha \ol{w}^\alpha I_{\fH_w}+A_\alpha A_\alpha^* ).
    \]
    Recall now that $T$ is a $K$-contraction, so that
	\begin{align*}
	 \sum_{n=1}^\infty b_n\sum_{|\alpha|=n}\frac{n!}{\alpha!} T^\alpha T^{*\alpha}\leq I_{\fH}.
	\end{align*}
	On the other hand, since $\|w\|=1$ we see that $wI_{\fH_w}$ is a spherical unitary and thus
	\[  \sum_{n=1}^\infty b_n\sum_{|\alpha|=n}\frac{n!}{\alpha!} w^\alpha \ol{w}^\alpha I_{\fH_w}=\sum_{n=1}^\infty b_n I_{\fH_w}=I_{\fH_w}. \]
	We conclude that
	\[
	 \sum_{n=1}^\infty b_n\sum_{|\alpha|=n}\frac{n!}{\alpha!}A_\alpha A_\alpha^* \leq 0
	\]
whence $b_1 A_j A_j^*=0$ for every $j=1,\ldots d$.
	Since $b_{1}=a_1>0$, we see that $A_jA_j^*=0$ for each $j=1,\ldots,d$ and thus $\fH_w$ is a reducing subspace on which $T$ restricts to be a spherical unitary. By the discussion preceding the lemma, this contradicts the fact that $T$ is $K$-asymptotically vanishing.
\end{proof}

When $\H$ is the Drury--Arveson space, it is well known that a commuting $K$-contraction is $K$-pure if and only if it is $K$-asymptotically vanishing \cite[Theorem 4.5]{arveson1998}. Because the model $M_x$ is always $K$-asymptotically vanishing (and this property is preserved by compressions to co-invariant subspaces), so are all $K$-pure commuting $K$-contractions. On the other hand, for non-maximal spaces all spherical unitaries are $K$-asymptotically vanishing yet they are not $K$-pure. In the maximal setting however, we recover the usual equivalence.

\begin{theorem}\label{T:pureKcontr}
	Let $\H$ be a maximal  regular  unitarily invariant space with kernel $K$, and let $T=(T_1,\ldots,T_d)$ be a commuting $K$-contraction on some Hilbert space $\fH$. Then, $T$ is $K$-pure if and only if it is $K$-asymptotically vanishing.
\end{theorem}
\begin{proof}
Invoking Theorem \ref{T:KCntrcFacts}, we obtain two Hilbert spaces $\fE$ and $\frk{K}$, a spherical unitary $U$ on $\frk{K}$, and an isometry $V:\fH\to(\H\otimes \fE)\oplus \fK $ such that 
    $
     VT^*=S^*V
    $
where we put $S=(M_x\otimes I_\fE)\oplus U$. We note that $S, M_x\otimes I_\fE$ and $U$ are all commuting $K$-contractions. For each $\alpha\in \bN^d$ we obtain
\begin{align*}
  T^\alpha T^{*\alpha}&=T^\alpha V^*V T^{*\alpha}=V^*S^\alpha S^{*\alpha}V
\end{align*}
whence
	\begin{align*}
		\Phi_T^{\of m}(I_\fH) &= V^*\Phi_{S}^{\of m}(I)V \\
											&=  V^*( (\Phi_{M_x}^{\of m}(I_\H)\otimes I_{\frk{E}})\oplus \Phi_{U}^{\of m}(I_{\fK} ))V 
	\end{align*}
for every $m\geq 1$. Since $U$ is a spherical unitary and $\H$ is maximal, we have $\Phi_{U}^{\of m}(I_{\fK} )=I_\fK$ for every $m\geq 1$. Hence
\[
 \Phi_T^{\of m}(I_\fH) =V^*( (\Phi_{M_x}^{\of m}(I_\H)\otimes I_{\frk{E}})\oplus I_\fK ))V, \quad m\geq 1.
\]
On the other hand, we know that $M_x$ is $K$-asymptotically vanishing.  Hence, for $h\in\frk{H}$ we have
	\begin{align*}
		\lim_{m\to\infty}\ip{\Phi_T^{\of m}(I_\fH)h,h} &= \|P_{\fK} Vh\|^2.
	\end{align*}
	Thus $T$ is $K$-asymptotically vanishing if and only if $\ran V\perp \fK$. This is easily seen to imply the desired equivalence.
\end{proof}

\subsection{Compatibility of the functional calculi}
Presently, for a given commuting $K$-contraction $T$, there a three functional calculi that could be available: the $\A(\H)$-functional calculus, the $\M(\H)$-functional calculus, and the Taylor functional calculus. In this subsection, we check that these all agree whenever they are simultaneously well defined. The first step is the following.

\begin{theorem}
    Let $\H$ be a regular unitarily invariant space with kernel $K$. Let $T=(T_1,\ldots,T_d)$ be a commuting $K$-contraction and let $U\subset \bC^d$ be an open neighbourhood of the closed unit ball. Let $f\in \O(U)$. Then, we have that $f|_{\bB_d}\in \A(\H)$ and
	$ \tau_{U,T}(f)=\alpha_T(f|_{\BB_d}). $
	\label{T:TandAHFCAgree} 
\end{theorem}
\begin{proof}
    It follows from Theorem \ref{T:KCntrcFacts} that $\TSpec(M_x)$ and $\TSpec(T)$ are contained in the closed unit ball. We may thus follow the proof of \cite[Theorem 2.6 (i)]{CTInterp} verbatim.    
\end{proof}

Let $f\in \M(\H)$. For each $0\leq r<1$, we define a function $f_r$ as
\[
 f_r(z)=f(rz), \quad z\in \bB_d.
\]
It then follows from Theorem \ref{T:TandAHFCAgree} that $f_r\in \A(\H)$ for every $0\leq r<1$. Furthermore, standard harmonic analysis techniques using the Poisson kernel show that $\|M_{f_r}\|\leq \|M_f\|$ for every $0\leq r<1$ and that the net $(M_{f_r})_{0\leq r<1}$ converges to $M_f$ in the strong operator topology of $B(\H)$ as $r$ increases to $1$. In particular, we conclude that $\A(\H)$ (and hence the polynomials) are weak-$*$ dense in $\M(\H)$. See \cite[Subsection 3.5]{shalit2014} for more details; the discussion therein readily adapts to our setting.

We can now show that the different functional calculi are compatible. The following extends \cite[Theorem 2.6 (ii)]{CTInterp}.

\begin{theorem}
	Suppose $\H$ is a regular unitarily invariant space with kernel $K$ and that $T=(T_1,\ldots,T_d)$ is an absolutely continuous commuting $K$-contraction.
	Let $U\subset \bC^d$ be an open neighborhood of $\TSpec(T)\cup \bB_d$, and let $f\in \O(U)$ be such that $f|_{\BB_d}\in\MH$.
	Then
	$\tau_{T,U}(f)=\widehat{\alpha}_T(f|_{\bB_d}). $
	\label{T:TnMHFCAgree}
\end{theorem}
\begin{proof}
	An elementary compactness argument yields a number $\delta>0$ with the property that if we let $V$ be the bounded open set consisting of those points $z\in U$ such that $\mathrm{dist}(z,\TSpec(T))<\delta$, then $\ol{V}\subset rU$ whenever $1\leq r<1+\delta$. 	The same argument then yields a number $0<\delta'<\delta$ with the property that if we let $W$ be the open set consisting of those points $z\in V$ such that $\mathrm{dist}(z,\TSpec(T))<\delta'$, then $\ol{W}\subset rV$ whenever $1\leq r<1+\delta'$. Note in particular that $W\subset V$ and that both are open neighbourhoods of $\TSpec(T)$.

		For an integer $n>1$, we define a function $f_n$ on $(1-1/n)^{-1}U$ by putting 
		\[f_n(z)
			=f((1-1/n)z)
		\]
		for $z\in (1-1/n)^{-1}U$.
		We also set $\psi=f|_{\bB_d}$ and $\psi_n=f_n|_{\bB_d}$. By assumption, $\psi\in \M(\H)$ so by Theorem \ref{T:TandAHFCAgree} we find that $\psi_n\in \A(\H)$ and 
	\[ \alpha_T(\psi_n)=\tau_{T,B_n}(f_n|_{B_n})\]
	where $B_n=(1-1/n)^{-1}\BB_d$.
	When $n$ is so large that $(1-1/n)^{-1}<1+\delta$, we see that $\cc{V}\subset (1-1/n)^{-1}U$ and hence $f_n$ is defined on $\ol{V}$. 	It follows from part (i) of Theorem \ref{T:Taylorprop} that
	\[ \tau_{T,V}(f_n|_{V})=\tau_{T,B_n}(f_n|_{B_n})=\alpha_T(\psi_n)\]
	for $n$ large enough.
	We may now use the remark preceding the theorem to conclude that the sequence $(M_{\psi_n})_n$ converges to $M_\psi$ in the weak-$*$ topology. But
	$T$ is absolutely continuous, so that the sequence $(\alpha_T(\psi_n))_n$ converges to $\widehat{\alpha}_T(\psi)$ in the weak-$*$ topology. Moreover,  another application of part (i) of Theorem \ref{T:Taylorprop} shows that $\tau_{T,U}(f)=\tau_{T,V}(f|_V)$. Thus, the proof will be complete once we show that the tail of the sequence $(\tau_{T,V}(f_n|_{V}))_n$ converges in norm to $\tau_{T,V}(f|_V)$.

    By the properties of the Taylor functional calculus (see Subsection \ref{SS:multispectra}) there is a constant $C>0$ such that
	\[ \|\tau_{T,V}(g)\|\leq C\sup_{z\in W}|g(z)| \]
	for every $g\in\mc{O}(V)$. When $n$ is so large that $1/(1-1/n)<1+\delta'$, we have that
		\[ \|\tau_{T,V}(f)-\tau_{T,V}(f_n)\|\leq C\cdot\sup_{z\in W}|f(z)-f_n(z)| \]
    and $(1-1/n)\ol{W}\subset V$. Since $f$ is uniformly continuous on the compact subset $\ol{V}$, a routine argument reveals that the sequence $(f_n|_W)_n$ converges uniformly to $f|_W$, and the proof is complete.
\end{proof}

In light of the previous theorem, we may (and will) write $f(T)$ in place of $\alpha_T(f)$, $\widehat{\alpha}_T(f)$, or $\tau_{T,U}(f)$, as appropriate.

\section{The spectra of functional models}\label{S:Spec} 

\subsection{The support of an ideal}\label{SS:supp}
The goal of this subsection is to identify the various spectra introduced in Subsection \ref{SS:multispectra} for the functional model $\kZa$. Our descriptions are based partly on the zero set of the ideal $\fa$. Before proceeding, we introduce some notation. Given a family $\F$ of functions defined on some set $E\subset \cc{\BB}_d$, we define the \textit{zero set} of $\F$ in $E$ as
\[
 \Z_E(\F)=\{z\in E:f(z)=0\text{ for all }f\in\mc{F}\}.
\]
For our purposes, $E$ will always be either $\bB_d$ or $\ol{\bB}_d$. 

\begin{theorem}\label{T:SpecInZ}
 	Let $\H$ be a regular unitarily invariant space with kernel $K$. Let $T=(T_1,\dots,T_d)$ be a commuting $K$-contraction. Then, we have that
 	\[
 	 \TSpec(T) \subset \grs'(T) \subset\wSpec(T)\subset \ZcBd(\Ann(T)\cap \A(\H)).
 	\]
    If $T$ is absolutely continuous, then
 	\[	\TSpec(T)\cap\BB_d \subset \grs'(T)\cap\BB_d \subset \wSpec(T)\cap\BB_d \subset  \ZBd(\Ann (T)). \]
\end{theorem}
\begin{proof}
The inclusions $\TSpec(T) \subset \grs'(T) \subset\wSpec(T)$ always hold.
 Let $f\in \Ann(T)\cap \A(\H)$ and let $\chi$ be a character of the weak-$*$ closed unital algebra generated by $T$. By part (iii) of Theorem \ref{T:KCntrcFacts}, we see that $\chi(T)\in \ol{\bB}_d$. Since the $\AH$-functional calculus is norm continuous, it follows that \[f(\chi(T))=\chi(f(T))=0.\] Hence, $\chi(T)\in \ZcBd(\Ann(T)\cap \A(\H))$ and we conclude that
 \[
 	 \wSpec(T)\subset \ZcBd(\Ann(T)\cap \A(\H)).
 	\]
 Assume now that $T$ is absolutely continuous and let $z\in\BB_d\bksl\ZBd(\Ann (T))$.
	Then there is $f\in\Ann(T)$ such that $f(z)\neq 0$.
	By Theorem \ref{T:GleasonTrick}, there are $f_1,\dots,f_d\in \mc{M}(\mc{H})$ such that $f=f(z)+\sum_{j=1}^d(x_j-z_j)f_j$, and thus
	\[ I=\sum_{j=1}^d (T_j-z_jI) (-f(z)^{-1}f_j(T)). \]
	Because the operators $f_1(T),\dots,f_d(T)$ lie in the weak-$*$ closed unital algebra generated by $T$, we conclude that $z\in\BB_d\bksl\wSpec(T)$. This shows that \[\wSpec(T)\cap\BB_d \subset  \ZBd(\Ann (T)). \qedhere \]
\end{proof}

The first set of inclusions above is not always very informative. This is best illustrated in the classical situation, as follows.

\begin{example}\label{E:AHAnn}
 Let $\H$ be the Hardy space on the unit disc. In this setting, a $K$-contraction is just a usual contraction. Accordingly, let $T$ be an absolutely continuous contraction. Let $\theta$ be a Blaschke product such that the cluster set of its zeros contains the unit circle and assume that $\Ann(T)=\theta \M(\H)$. By continuity, we see that $\Ann(T)\cap \A(\H)=\{0\}$. Thus, in this case the first set of inclusions in Theorem \ref{T:SpecInZ} is trivial. 
 \qed
\end{example}

Example \ref{E:AHAnn} indicates that Theorem \ref{T:SpecInZ} does not adequately describe the boundary portions of the spectra in general. As can be inferred from the proof, one reason for this is that there is no version of Theorem \ref{T:GleasonTrick} for points on the sphere. The following allows us to tackle this problem.

Let $\fa\subset \M(\H)$ be an ideal. We define its \textit{support} to be the set consisting of all those $z\in\CC^d$ for which no element $f\in\frk{a}$ admits a decomposition of the form $f=c+\sum_{j=1}^d(x_j-z_j)f_j$ for $c\in\CC\bksl\{0\}$ and $f_1,\dots,f_d\in\mc{M}(\mc{H})$.
More concisely, we have that $z\in \supp(\fa)$ if and only if
\[  1\nin \frk{a}+\sum_{j=1}^d(x_j-z_j)\mc{M}(\mc{H}). \]
An alternative description of the support goes as follows.

\begin{theorem}
    Let $\H$ be a regular unitarily invariant space and let $\frk{a}$ be a weak-$*$ closed ideal of $\MH$. Let $z\in \bC^d$. Then, $z\nin\supp(\frk{a})$ if and only if there exist $f\in\frk{a}$, $c\in\CC\bksl\{0\}$ and $L>0$ such that
	\[ M_{f-c}M_{f-c}^* \leq L^2 M_{x-z}M_{x-z}^*. \]
	In this case, we have $|f(w)-c|\leq L \|w-z\|$ for every $w\in\BB_d$.
	\label{T:DfLTDx}
\end{theorem}
\begin{proof}
	Suppose first that $z\nin\supp(\frk{a})$. There exist $f\in\frk{a}$, $c\in\CC\bksl\{0\}$ and $f_1,\dots,f_d\in\mc{M}(\mc{H})$ such that $f-c=\sum_{j=1}^d (x_j-z_j)f_j$.
	In particular, we find  $M_{f-c}=M_{x-z} F$ where
	\[ x-z=\begin{bmatrix} x_1-z_1 & \ldots & x_d-z_d\end{bmatrix}\in \M(\H\otimes\bC^d,\H)\]
	and 
	\[
	F=\begin{bmatrix} f_1 \\ \vdots \\ f_d
	       \end{bmatrix}\in \M(\H,\H\otimes \bC^d).
    \]
	Thus
	\[ M_{f-c}M_{f-c}^*=M_{x-z}M_{F}M_{F}^*M_{x-z}^*\leq \|M_{F}\|^2M_{x-z}M_{x-z}^*. \]
Moreover, using the Cauchy-Schwarz inequality we obtain
	\begin{align*}
	|f(w)-c|&\leq \sum_{j=1}^d|w_j-z_j| |f_j(w)| \leq \|F(w)\| \|w-z\|\\
	&\leq \|M_F\|\|w-z\|
	 	\end{align*}
for every $w\in \bB_d$.

Conversely, suppose that there exists $f\in\frk{a}$, $c\in\CC\bksl\{0\}$ and $L>0$ such that $M_{f-c}M_{f-c}^*\leq L^2 M_{x-z}M_{x-z}^*$.
By \cite[Theorem 8.57]{agler2002}, we find
	\[
	 	F=\begin{bmatrix} f_1 \\ \vdots \\ f_d
		\end{bmatrix}\in \M(\H,\H\otimes \bC^d)
	\]
	such that $\|M_F\|\leq L$ and $f-c=\sum_{j=1}^d (x_j-z_j)f_j$. Thus $z\nin\supp(\frk{a})$.
\end{proof}

We next establish some basic properties of the support.

\begin{theorem}\label{T:suppprop}
 	Let $\H$ be a regular unitarily invariant space. The following statements hold.
 	\begin{enumerate}[{\rm (i) }]
 	 \item Let $\frk{a}\subset \mc{M}(\mc{H})$ be an ideal. Then we have	 	 $\supp(\frk{a})\cap\BB_d=\ZBd(\frk{a})$
 	 and $\supp(\frk{a})\subset\cc{\BB}_d.$
 	 
 	 \item 	Let $\frk{a},\frk{b}\subset \mc{M}(\mc{H})$ be ideals such that $\frk{a}\subset \frk{b}$. Then $\supp(\frk{b})\subset\supp(\frk{a})$.
	
	\item Let $\frk{a},\frk{b}\subset \M(\H)$ be ideals.
	Then
	\[ \supp(\frk{a}\cap\frk{b})=\supp(\frk{a})\cup\supp(\frk{b})=\supp(\frk{a}\cdot\frk{b}), \]
	and
	\[\supp(\frk{a}+\frk{b})\subset \supp(\frk{a})\cap \supp(\frk{b}). \]
	
	\item  Let $\frk{a},\frk{b}\subset \M(\H)$ be weak-$*$ closed ideals. Then
	\[ \supp(\frk{a}\cdot\frk{b})=\supp(\cc{\frk{a}\cdot\frk{b}}^{w*}). \]
 	\end{enumerate}

\end{theorem}
\begin{proof}
(i) Let $z\in\BB_d\bksl\ZBd(\frk{a})$ and choose $f\in\frk{a}$ such that $f(z)\neq 0$. Use now Theorem \ref{T:GleasonTrick} to find $f_1,\dots,f_d\in\mc{M}(\mc{H})$ such that $f=f(z)+\sum_{j=1}^d (x_j-z_j)f_j$. By definition, we find $z\nin\supp\frk{a}$. Conversely, let $z\in\BB_d\bksl\supp\frk{a}$. There is $c\in\CC\bksl\{0\}$ and $g_1,\cdots,g_d\in\mc{M}(\mc{H})$ such that the function $g=c+\sum_{j=1}^d(x_j-z_j)g_j$ lies in $\frk{a}$.
	In particular, $g(z)=c\neq 0$ and so $z\nin\ZBd(\frk{a})$. We conclude that $\supp(\frk{a})\cap\BB_d=\ZBd(\frk{a})$.

Next, let $z\in\CC^d$ with $\|z\|>1$. By the Cauchy--Schwarz inequality we see that the function $1/\ip{x-z,z}$ is analytic on $\{w\in \bC^d: \|w\|<\|z\|\}$, which is an open neighbourhood of the closed unit ball. Thus, by Theorem \ref{T:TandAHFCAgree} we see that $1/\ip{x-z,z}\in\mc{M}(\mc{H})$. For any $f\in\frk{a}$ we find
	\[ f=1+\ip{x-z,z}\frac{f-1}{\ip{x-z,z}}=1+\sum_{j=1}^d (x_j-z_j)\frac{\cc{z}_j (f-1)}{\ip{x-z,z}} \]
	whence $z\nin\supp(\frk{a})$. We conclude that $\supp(\fa)\subset \ol{\bB}_d$.

(ii) This is a straightforward consequence of the definition.

(iii) Because
	\[ \frk{a}\cdot \frk{b}\subset \frk{a}\cap\frk{b} \subset \frk{a},\frk{b}\subset \frk{a}+\frk{b} \]
	it follows from (ii) that
	\[ \supp(\frk{a})\cup\supp(\frk{a_2})\subset \supp(\frk{a}\cap\frk{b})\subset \supp(\frk{a}\cdot\frk{b}), \]
	and that
	\[ \supp(\frk{a}+\frk{b})\subset \supp(\frk{a})\cap\supp(\frk{b}). \]
	Let $z\in\cc{\BB}_d\bksl(\supp(\frk{a})\cup\supp(\frk{b}))$. There exist $f_1\in\frk{a}, f_2\in\frk{b}$, $c_1,c_2\in\CC\bksl\{0\}$ and $g_{1},\ldots,g_{d}$, $h_{1},\ldots,h_{d}\in\mc{M}(\mc{H})$ such that
	\[ f_1=c_1+\sum_{j=1}^d (x_j-z_j)g_{j}, \quad f_2=c_2+\sum_{j=1}^d (x_j-z_j)h_{j}. \]
	Thus
	\[	f_1 f_2 =  c_1c_2+\sum_{i=1}^d (x_i-z_i)\left(c_2g_{i}+c_1h_{i}+g_{i}\sum_{j=1}^d (x_j-z_j)h_{j}\right) \]
	whence $z\in\cc{\BB}_d\bksl\supp(\frk{a}\cdot\frk{b})$.
	Therefore $\supp(\frk{a}\cdot\frk{b})\subset \supp(\frk{a})\cup\supp(\frk{b})$ and (iii) follows.
	
(iv) If $\frk{a},\frk{b}$ are weak-$*$ closed, then so is $\fa\cap \frk{b}$ and hence
	\[ \frk{a}\cdot\frk{b}\subset \cc{\frk{a}\cdot\frk{b}}^{w*}\subset \frk{a}\cap\frk{b}. \]
	It follows from (ii) and (iii) that
	\[ \supp(\frk{a}\cap\frk{b})\subset \supp(\cc{\frk{a}\cdot\frk{b}}^{w*})\subset \supp(\frk{a}\cdot\frk{b})=\supp(\frk{a}\cap\frk{b}). \qedhere \]
\end{proof}

The previous theorem shows that in many ways, the support of an ideal behaves like its zero set. There is a wrinkle in this analogy however: if $\fa,\frk{b}\subset \M(\H)$ are ideals, then
\[ \ZBd(\cc{\frk{a}+\frk{b}}^{w*})=\ZBd(\frk{a}+\frk{b})=\ZBd(\frk{a})\cap\ZBd(\frk{b}) \]
whereas above we only established one inclusion of the corresponding equality for supports. It turns out that the reverse inclusion does not typically hold, even in the classical situation of the Hardy space on the disc, as the next example shows. The argument requires the fact that the support of an inner function coincides with the support of the ideal it generates; this will be proved in Theorem \ref{T:suppEqSpec} below.

\begin{example}\label{E:suppsuminclusionw*}
Let $\H$ be the Hardy space on the disc. 
Let $\phi,\psi\in \M(\H)$ be two Blaschke products with disjoint zero sets but with overlapping supports. Let $\fa=\phi \M(\H)$ and $\frk{b}=\psi\M(\H)$. 
Applying Theorem \ref{T:LatIsomThm} and Beurling's theorem to the invariant subspace $[(\fa+\frk{b})\H]$, we find an inner function $\theta\in \M(\H)$ such that the weak-$*$ closure of $\fa+\frk{b}$ coincides with $\theta \M(\H)$.
In particular, we have that $\phi,\psi\in \theta\M(\H)$, and a moment's thought reveals that $\theta$ must therefore be constant because of the choice of $\phi$ and $\psi$.
Thus, $\fa+\frk{b}$ is weak-$*$ dense in $\M(\H)$ so that $\supp(\cc{\frk{a}+\frk{b}}^{w*})$ is empty, while $\supp(\fa)\cap \supp(\frk{b})$ is not. 
\qed
\end{example}

In fact, even the weaker inclusion $\supp (\fa)\cap \supp(\frk{b})\subset\supp(\fa+\frk{b})$ can fail for general ideals.

\begin{example}\label{E:suppsuminclusion}
Once again, we take $\H$ to be the Hardy space on the disc. It is well known that there exist two distinct characters $\chi_1,\chi_2$ on $\M(\H)$ such that $\chi_1(x)=\chi_2(x)=1$ \cite[page 161]{hoffman1988}. Let $\fa=\ker \chi_1$ and $\fb=\ker\chi_2$, which are two distinct codimension one ideals in $\M(\H)$. Thus, $\fa+\fb=\M(\H)$ so that $\supp(\fa+\fb)$ is empty. On the other hand, we note that $x-1\in \fa\cap \fb$, whence $\fa=\fa+(x-1)\M(\H)$ and $\fb=\fb+(x-1)\M(\H)$. In particular, we conclude that $1\notin \fa+(x-1)\M(\H)$ and $1\notin \fb+(x-1)\M(\H)$, whence $1\in \supp(\fa)\cap \supp(\fb)$.
	\qed
\end{example}

Our next goal is to show that the support of a weak-$*$ closed ideal $\fa$ coincides with both the Taylor and the weak-$*$ spectra of the functional model $\kZa$. 

First, we prove an elementary inclusion which holds more generally. The reader will note that the proof is almost identical to that of Theorem \ref{T:SpecInZ}, now that the support is defined appropriately. 

\begin{lemma}
	Let $\H$ be a regular unitarily invariant space with kernel $K$, and let $T=(T_1,\dots,T_d)$ be an absolutely continuous commuting $K$-contraction. 
	Then
	$ \wSpec(T)\subset \supp(\Ann (T)). $
	\label{L:wSpecInSupp}
\end{lemma}
\begin{proof}
	Suppose $z\in\CC^d\bksl\supp(\Ann (T))$. There are $f\in\Ann (T)$,  $c\in\CC\bksl\{0\}$, and $f_1,\dots,f_d\in\mc{M}(\mc{H})$ such that $f=c+\sum_{j=1}^d(x_j-z_j)f_j$.
	Therefore
	\[ 0=f(T)=cI+\sum_{j=1}^d (T_j-z_jI)f_j(T). \]
	Since $c\neq 0$, we have $z\nin\wSpec(T)$.
\end{proof}

We can now prove one of the main results of the paper.

\begin{theorem}
Let $\H$ be a regular unitarily invariant space and let $\frk{a}$ be a weak-$*$ closed ideal of $\MH$.
	Then
	\[ \TSpec(\kZa)=\wSpec(\kZa)=\supp(\frk{a}). \]
	\label{T:suppEqSpec}
\end{theorem}
\begin{proof}
    By Theorem \ref{T:KCntrcFacts}, we see that $\kZa$ is an absolutely continuous commuting $K$-contraction, whence it follows from Lemma \ref{L:wSpecInSupp} that
	\[ \TSpec(\kZa)\subset \wSpec(\kZa) \subset \supp(\frk{a}). \]
	It thus remains to prove that $\supp(\fa)\subset \TSpec(\kZa)$. 
	
	We first deal with the interior points. Let $z\in\supp(\frk{a})\cap\BB_d$.
	By Theorem \ref{T:suppprop}, we see that $z\in\ZBd(\frk{a})$ whence $k_z\in \mc{H}(\fa)$, and so
	\[ Z^{(\frk{a})*}_jk_z= M_{x_j}^*k_z = \cc{z}_jk_z \]
	for $j=1,\dots,d$. By virtue of \eqref{Eq:incTayPoint}, we find that $\cc{z}\in \TSpec(Z^{(\frk{a})*})$ and hence $z\in\TSpec(\kZa)$ by \eqref{Eq:TayAdj}.
	Therefore $\supp(\frk{a})\cap\BB_d= \TSpec(\kZa)\cap\BB_d$.
    
    We now move on to the boundary points. Let $z\in\pd\BB_d\bksl\TSpec(\kZa)$.
	The function $g=1/\ip{x-z,z}$ is analytic on some neighborhood $U$ of $\TSpec(\kZa)$ and satisfies
	$
    1=\sum_{j=1}^d (x_j-z_j)\ol{z}_j g 
	$
	on $U$. 	By applying the Taylor functional calculus, we obtain an operator $g(\kZa)\in \{\kZa\}''$ such that
	\[ I-\sum_{j=1}^d (\kZa_j-z_jI)\cc{z}_jg(\kZa)=0. \]
	Clearly, we have that $g(\kZa)\in \{\kZa\}'$ so we may apply the commutant lifting theorem \cite[Theorem 5.1]{BTV2001} to obtain $f\in \M(\H)$ with $g(\kZa)=f(Z^{(\frk{a})})$. Thus
	\[ 0= I-\sum_{j=1}^d (\kZa_j-z_jI)\cc{z}_jf(\kZa) \]
	and  by Theorem \ref{T:QCIsom} we infer that $1-\sum_{j=1}^d (x_j-z_j)\cc{z}_jf \in\frk{a}$. This says precisely that $z\nin\supp(\frk{a})$.
\end{proof}

\subsection{Supports and approximate zero sets}\label{SS:suppAZ}
In this subsection, we aim to give more concrete descriptions of the support of an ideal. For this purpose, we need the following notion. 
Given a family $\F$ of functions on $\BB_d$, we define its \textit{approximate zero set} to be the set $\AZ(\mc{F})$ consisting of all those $z\in\cc{\BB}_d$ for which there is a sequence $(w_n)_n$ in $\bB_d$ that converges to $z$ and such that $(f(w_n))_n$ converges to $0$ for every $f\in \F$.

\begin{theorem}\label{T:suppEqAZ}
    Let $\H$ be a regular unitarily invariant space with the Corona property and let $\frk{a}$ be a weak-$*$ closed ideal of $\mc{M}(\mc{H})$. Then, we have
	$ \supp(\frk{a})=\mathrm{AZ}(\frk{a}). $
	
\end{theorem}
\begin{proof}
	Let $z\nin\supp(\frk{a})$. There exist $f\in \fa$, $c\in\CC\bksl\{0\}$ and $f_1,\ldots,f_d\in \M(\H)$ such that
	$f=c+\sum_{j=1}^d (x_j-z_j) f_j$. This formula shows that if $(w_n)_n$ is a sequence in $\bB_d$ that converges to $z$, then $(f(w_n))_n$ converges to $c\neq 0$.  Thus $z\nin \AZ(\fa)$. We conclude that $\AZ(\frk{a})\subset\supp(\frk{a})$.

	Conversely, suppose $z\in\supp(\frk{a})$.
	It follows from Theorem \ref{T:suppEqSpec} that $z\in\wSpec(\kZa)$ and thus there is a character $\chi$ of $\B_\fa$ such that $\chi(\kZa)=z$.
	By Theorem \ref{T:QCIsom}, there is a character $\wtil{\chi}$ of $\mc{M}(\mc{H})$ such that $\wtil{\chi}(M_x)=z$ and $\fa\subset \ker \wtil{\chi}$. Now, by the Corona property there is a sequence $(w_n)_n$ in $\BB_d$ such that $(f(w_n))_n$ converges to $\wtil{\chi}(f)$ for every $f\in\mc{M}(\mc{H})$. In particular, we see that $(w_n)_n$ converges to $z$ while $(f(w_n))_n$ converges to $0$ for every $f\in \fa$. This says precisely that $z\in \AZ(\fa)$, so that $\supp(\fa)\subset \AZ(\fa)$.
\end{proof}

For some nicely behaved ideals, the approximate zero set in the previous description can be replaced by a genuine zero set.

\begin{corollary}\label{C:AHDense}
    Let $\H$ be a regular unitarily invariant space with the Corona property and let $\frk{a}$ be a weak-$*$ closed ideal of $\mc{M}(\mc{H})$.
	Then
	\[ \supp(\frk{a})\subset\ZcBd(\frk{a}\cap\AH). \]
	If, in addition, $\frk{a}\cap\AH$ is assumed to be weak-$*$ dense in $\frk{a}$ and $\ZcBd(\frk{a}\cap\AH)=\cc{\ZBd(\frk{a}\cap\AH)}$, then
	\[ \supp(\frk{a})=\ZcBd(\frk{a}\cap\AH). \]
\end{corollary}
\begin{proof}
	Let $z\in\supp(\frk{a})$.
	By Theorem \ref{T:suppEqAZ}, there is a sequence $(w_n)_n$ in $\BB_d$ converging to $z$ such that $(f(w_n))_n$ converges to $0$ for every $f\in\frk{a}$. If $g\in\frk{a}\cap\AH$, then by continuity we find $g(z)=0$. Hence we have $z\in\ZcBd(\frk{a}\cap\AH)$.

	For the remainder of the proof we suppose that $\frk{a}\cap \AH$ is weak$*$ dense in $\frk{a}$ and that 
	\[\ZcBd(\frk{a}\cap\AH)=\cc{\ZBd(\frk{a}\cap\AH )}.
	 \]
    Using the appropriate kernel vector, we see that evaluation at a point in the interior of the ball is a weak-$*$ continuous functional on $\M(\H)$, and thus 
    \[
     \ZBd(\frk{a})=\ZBd(\frk{a}\cap\AH).
     \]
	Note now that $\ZBd(\frk{a})\subset\supp(\frk{a})$ by Theorem \ref{T:suppprop}, and that $\supp(\frk{a})$ is closed as a consequence of Theorem \ref{T:suppEqSpec}.
	Thus
	\[ \ZcBd(\frk{a}\cap\AH)=\cc{\ZBd(\frk{a}\cap\AH)}\subset\supp(\frk{a})\subset \ZcBd(\frk{a}\cap\AH ) \]
	where the last inclusion follows from the previous paragraph.
\end{proof}

Homogeneous ideals fit into the framework of the previous result, as we illustrate next.

\begin{example}\label{E:HomogPoly}
Let $\H$ be a regular unitarily invariant space with the Corona property. As explained in the discussion following Theorem \ref{T:TandAHFCAgree}, the polynomials are weak-$*$ dense in $\M(\H)$.	

Let $p_1,\dots,p_m\in\CC[x_1,\dots,x_d]$ be homogeneous polynomials and let $\frk{a}$ be the weak-$*$ closed ideal of $\mc{M}(\mc{H})$ that they generate. 	
Every element of $\frk{a}$ is a weak-$*$ limit of elements from the ideal generated by $p_1,\ldots,p_m$ inside of $\CC[x_1,\dots,x_d]$, and thus 
\[
\ZBd(\fa)=\ZBd(\frk{a}\cap\AH)=\ZBd(p_1,\ldots,p_m).
\]
	Since $p_1,\ldots,p_m$ are homogeneous, we have $\ZcBd(p_1,\ldots,p_m)=\cc{\ZBd(p_1,\dots,p_m)}$, and thus
	\[ \ZcBd(p_1,\ldots,p_m)=\cc{\ZBd(p_1,\dots,p_m)}=\cc{\ZBd(\frk{a}\cap\AH)}. \]
	By continuity, we find
	\[ \cc{\ZBd(\frk{a}\cap\AH)}\subset \ZcBd(\frk{a}\cap\AH)\subset\ZcBd(p_1,\ldots,p_m) \]
	and so
	\[ \ZcBd(p_1,\ldots,p_m)=\ZcBd(\frk{a}\cap\AH)=\cc{\ZBd(\frk{a}\cap\AH)}. \]
	Thus
	\[ \supp(\frk{a})=\ZcBd(\frk{a}\cap\AH)=\ZcBd(p_1,\dots,p_m).  \]
	by Corollary \ref{C:AHDense}.
	\qed
\end{example}

    As explained in the introduction (see also \cite[Definition 2.2.15]{bercovici1988}), in the case of the Hardy space on the disc, the boundary portion of the support of an inner function is associated with a certain loss of regularity. A similar phenomenon occurs more generally, as we show next.

\begin{corollary}\label{C:Lipschitz}
 	Let $\H$ be a regular unitarily invariant space with the Corona property and let $\frk{a}$ be a weak-$*$ closed ideal of $\mc{M}(\mc{H})$. Let $z\in \bC^d$.  Then $z\nin\supp(\frk{a})$ if and only if there is $f\in\frk{a}$ which is (or extends to be)  Lipschitz and non-zero at $z$.
\end{corollary}
\begin{proof}
Assume first that $z\nin\supp(\frk{a})$. It follows from Theorem \ref{T:DfLTDx} that there is a function $f\in \fa$ which is (or extends to be)  Lipschitz and non-zero at $z$. Conversely, assume that $f\in \fa$
	is (or extends to be)  Lipschitz and non-zero at $z$. If $z\notin \ol{\bB}_d$, then $z\notin \supp(\fa)$ by Theorem \ref{T:suppprop}. Assume therefore that $z\in \ol{\bB}_d$. Let $(w_n)_n$ be a sequence in $\bB_d$ converging to $z$. We see that $(f(w_n))_n$ converges to $f(z)\neq 0$. Since $f\in \fa$, this means that $z\notin \AZ(\fa)$.
	Invoking Theorem \ref{T:suppEqAZ} we see that $z\nin\supp(\frk{a})$.
\end{proof}

To close this section, we have the following result that provides a description of the Banach algebraic spectrum of functions of $\kZa$ in $\B_{\frk{a}}$ as a common approximate zero set.

\begin{theorem}
	Suppose $\H$ is a regular unitarily invariant space with the Corona property, and fix a weak-$*$ closed ideal $\frk{a}$ of $\M(\H)$.
	Given $u_1,\dots,u_s\in\mc{M}(\mc{H})$ and $z\in\CC^s$, the following statements are equivalent:
	\begin{enumerate}
		\item[{\rm{(i)}}] $z\in\sigma_{\B_{\frk{a}}}(u_1(\kZa),\dots,u_s(\kZa))$,
		\item[{\rm{(ii)}}] 

			\[ \sup_{f\in\frk{a}}\inf_{w\in\BB_d} \left\{|f(w)|+\sum_{j=1}^s |u_j(w)-z_j|\right\} = 0.\]
	\end{enumerate}
	\label{L:SpecAndTheta}
\end{theorem}

\begin{proof}
(ii) $\Rightarrow$ (i): Assume that $z\nin\sigma_{\mc{B}_{\frk{a}}}(u_1(\kZa),\dots,u_s(\kZa))$. By definition, this means that there exist $f_1,\dots,f_s\in\mc{M}(\mc{H})$ such that
	\[ \sum_{j=1}^s (u_j(\kZa)-z_j I)f_j(\kZa)=I. \]
	By Theorem \ref{T:QCIsom}, we see that the function $f=1-\sum_{j=1}^s (u_j-z_j)f_j$ lies in $\fa$.
	For any $w\in\BB_d$, we observe that
	\begin{align*}
	 1&\leq |f(w)|+\sum_{j=1}^s \|f_j\||u_j(w)-z_j|
	\end{align*}
	whence
	\[
	 \inf_{w\in\BB_d} \left\{|f(w)|+\sum_{j=1}^s |u_j(w)-z_j|\right\}>0.
	\]

	(i) $\Rightarrow$ (ii): Suppose that
	\[ \inf_{w\in\BB_d}\left\{ |f(w)|+\sum_{k=1}^s|u_k(w)-z_k| \right\}>0 \]
	for some $f\in \fa$. Because $\H$ has the Corona property, there are $g_0,g_1,\ldots,g_s\in\MH$ such that
	\[ 1=g_0f+(u_1-z_1)g_1+\cdots+(u_s-z_s)g_s. \]
	Since $f\in \fa$, we have $f(\kZa)=0$ and thus \[I=\sum_{j=1}^s(u_j(\kZa)-z_j)g_j(\kZa).\] 
	This equality implies that $z\nin\sigma_{\mc{B}_{\frk{a}}}(u_1(\kZa),\ldots,u_s(\kZa))$.
\end{proof}

The condition that $\H$ has the Corona property is necessary in the previous result. Indeed, the equivalence of (i) and (ii) for the ideal $\fa=\{0\}$ says precisely that $\H$ has the Corona property. 

The preceding theorem allows for the following, somewhat surprising, characterization of the support.
In particular, it says that the approximate zero set of a weak-$*$ closed ideal $\frk{a}$ consists of all those approximate zeroes which are common to the elements of $\frk{a}$.

\begin{corollary} \label{T:suppAZTheta}
Let $\H$ be a regular unitarily invariant space with the Corona property and let $\frk{a}$ be a weak-$*$ closed ideal of $\mc{M}(\mc{H})$.
Then,
	\[ \supp(\frk{a})=\bigcap_{f\in\frk{a}}\AZ(f). \]
\end{corollary}
\begin{proof}

	Plainly $\AZ(\frk{a})\subset\bigcap_{f\in\frk{a}}\AZ(f)$. To see the reverse inclusion let  $z\in\bigcap_{f\in\frk{a}}\AZ(f)$. Taking $s=d$, $u_1=x_1,\ldots,u_d=x_d$ in Lemma \ref{L:SpecAndTheta} and invoking Theorem \ref{T:suppEqSpec}, it follows that $z\in\sigma_{\mc{B}_{\frk{a}}}(\kZa)=\supp(\fa)$. 	
\end{proof}

By analogy with the classical univariate case in the Hardy space, the reader may wonder  whether the support can be characterized in terms of vanishing information of a single operator-valued inner function.
Indeed, one can show that, given a weak-$*$ ideal $\fa$ of $\M(\H)$, there exist a separable Hilbert space $\frk{E}$ and an inner multiplier $\Theta\in\M(\H\otimes\frk{E},\H)$ such that $\frk{a}=\Theta\M(\H,\H\otimes\frk{E})$.
How does $\supp(\frk{a})$ relate to the properties of $\Theta$? 
One can readily show that $\supp(\frk{a})\supset \AZ(\|\Theta(\cdot)\|)$ 
and that \[\AZ(\{\Theta(\cdot)v:v\in\frk{E}\})\supset \supp(\frk{a}),\] 
but the equality of $\AZ(\{\Theta(\cdot)v:v\in\frk{E}\})$ with $\AZ(\|\Theta(\cdot)\|)$ is not clear. We will not pursue these ideas further in this paper.

\section{Localizable points in the support}\label{S:Loc} 

Let $\H$ be a regular unitarily invariant space with kernel $K$. Recall that an absolutely continuous commuting $K$-contraction  $T=(T_1,\dots,T_d)$ satisfies 
$ \wSpec(T)\subset \supp(\Ann( T))$, by Lemma \ref{L:wSpecInSupp}. In this section, we explore the reverse inclusion with the help of the following definition, wherein $B(z,r)$ denotes the open ball of radius $r$ centred at some point $z\in\CC^d$.

	Let $\frk{a}$ be a weak-$*$ closed ideal of $\MH$.
	Given $z\in\CC^d$, we say that $\frk{a}$ is \textit{localizable} at $z$ if, for each $r>0$, there exist weak-$*$ closed ideals $\frk{b},\frk{c}\subset\MH$ with the following properties:
	\begin{enumerate}[(i)]
		\item $\frk{c} \not\subset \frk{a}$,
		\item $\frk{b}\frk{c}\subset\frk{a}$,
		\item $\supp(\frk{a})\cap\supp(\frk{b})\subset B(z,r)$.
	\end{enumerate}
	The set of all $z\in\CC^d$ at which $\frk{a}$ is localizable is denoted by $\mathrm{Loc}(\frk{a})$. As we show next, this set is always a closed subset of the closed unit ball.

\begin{lemma}
    Let $\H$ be a regular unitarily invariant space and let $\frk{a}$ be a weak-$*$ closed ideal of $\MH$.
	Then $\mathrm{Loc}(\frk{a})$ is a closed subset of $\cc{\BB}_d$.
	\label{L:LocInClsdBall}
\end{lemma}
\begin{proof}
	Suppose $z\in\mathrm{Loc}(\frk{a})\bksl\cc{\BB}_d$ and let $r=(\|z\|-1)/2>0$. There exist weak-$*$ closed ideals $\frk{b},\frk{c}\subset\MH$ such that $\frk{c} \not\subset \frk{a}$, $\frk{b}\frk{c}\subset\frk{a}$, and $\supp(\frk{a})\cap\supp(\frk{b})\subset B(z,r)$. Recall now from Theorem \ref{T:suppprop} that the support of an ideal is always contained in $\cc{\BB}_d$. On the other hand, by choice of $r$ we see that $B(z,r)$ is disjoint from $\cc{\BB}_d$, which forces $\supp(\frk{a})\cap\supp(\frk{b})$ to be empty. Therefore, if we let $\frk{w}$ be the weak-$*$ closure of $\fa+\fb$, then we find that $\supp(\fw)$ is also empty by virtue of Theorem \ref{T:suppprop} again. In particular, Theorem \ref{T:suppEqSpec} implies that $\TSpec(Z^{(\fw)})$ is empty, and thus $\H(\fw)$ is the zero space (see Subsection \ref{SS:multispectra}). Invoking Theorem \ref{T:LatIsomThm} we find that $\fw=\M(\H)$. Using that $\fb \fc\subset \fa$, we find $(\fa+\fb)\fc\subset \fa$ and since $\fa$ is weak-$*$ closed we also get $\fc=\fw\fc\subset \fa$, contrary to the choice of $\fc$.
	We conclude that $\mathrm{Loc}(\frk{a})\subset\cc{\BB}_d$.

	It remains to show that $\mathrm{Loc}(\frk{a})$ is closed.
	Let $(z_n)_n$ be a sequence in $\mathrm{Loc}(\frk{a})$ that converges to some $z\in\cc{\BB}_d$.
	Let $r>0$ and suppose $n$ is some positive integer large enough that $\|z_n-z\|<r/2$.
	Since $z_n\in \Loc(\fa)$, there exist weak-$*$ closed ideals $\frk{b},\frk{c}\subset\MH$ such that $\frk{c} \not\subset\frk{a}$, $\frk{b}\frk{c}\subset\frk{a}$, and $\supp(\frk{a})\cap\supp(\frk{b})\subset B(z_n,r/2)$.
	Since $B(z_n,r/2)\subset B(z,r)$, we have $z\in\mathrm{Loc}(\frk{a})$.
\end{proof}

\subsection{Localizable points and spectra}

We can now prove another one of the main results of the paper, the proof of which is inspired by that of \cite[Theorem 2.4.11]{bercovici1988}. 

\begin{theorem}
    Let $\H$ be a regular unitarily invariant space with kernel $K$.
	Let $T=(T_1,\dots,T_d)$ be an absolutely continuous commuting $K$-contraction.
	Then $\mathrm{Loc}(\Ann (T))\subset\sigma'(T)$.
	\label{T:LocInSigPrime}
\end{theorem}
\begin{proof}
	Let $\frk{H}$ denote the Hilbert space on which $T_1,\ldots,T_d$ act and let $\frk{a}=\Ann(T)$.
	If $\sigma'(T)$ is empty, then $\frk{H}=\{0\}$ and $\Ann(T)=\MH$.
	We then have that $\Loc(\fa)$ is empty as well and the desired statement holds.

	Assume now that $\sigma'(T)$ is non-empty. We may clearly assume that $\Loc(\frk{a})$ is non-empty as well. Fix $z\in\Loc(\frk{a})$ and $r>0$. There exist weak-$*$ closed ideals $\frk{b},\frk{c}\subset\MH$ such that $\frk{c} \not\subset \frk{a}$, $\frk{b}\frk{c}\subset\frk{a}$, and $\supp(\frk{a})\cap\supp(\frk{b})\subset B(z,r)$. The fact that $\fc\not\subset \fa$ implies that the closed subspace \[\frk{M}=\ol{\spn\{\ran f(T):f\in \fc\}}\] is non-zero. Note also that $\fM$ is invariant for $\{T\}'$. Because $\fb\fc\subset \fa=\Ann(T)$, we have $\frk{b}\subset \Ann(T|_{\frk{M}})$. 
		By Theorem \ref{T:suppprop} and Lemma \ref{L:wSpecInSupp} we have
	\[ \sigma'(T|_{\frk{M}})\subset \supp(\Ann(T|_\fM))\subset \supp(\frk{b}). \]
	On the other hand, by \eqref{Eq:ComSpecRestrict} and Lemma \ref{L:wSpecInSupp} we have
	\[ \sigma'(T|_{\frk{M}})\subset \sigma'(T) \subset\supp(\frk{a}). \]
	Thus, 
	\[ \sigma'(T|_{\frk{M}}) \subset \supp(\frk{a})\cap\supp(\frk{b})\subset B(z,r), \]
	whence
	\[ \sigma'(T|_{\frk{M}})\subset \sigma'(T)\cap B(z,r). \]
	Now, $\sigma'(T|_{\frk{M}})$ is non-empty as $\fM$ is non-zero, and so  $\sigma'(T)\cap B(z,r)$ is non-empty.
	This holds for any $r>0$, and thus $z\in \sigma'(T)$ since $\sigma'(T)$ is closed.

\end{proof}

A statement stronger than the one above would be that $\Loc(\fa)\subset \TSpec(T)$. We do not know if this holds. One obstacle to a direct modification of the above proof is the failure of \eqref{Eq:ComSpecRestrict} for the Taylor spectrum, as shown in \cite{timko2019spec}.

Summarizing the spectral information contained in Lemma \ref{L:wSpecInSupp} and Theorem \ref{T:LocInSigPrime} for an absolutely continuous commuting $K$-contraction, we write
\[ \Loc(\Ann (T))\subset\sigma'(T)\subset \supp(\Ann (T)). \]
Next, we show that both of these inclusions can be strict in general.

\begin{example}\label{E:SmallSpecLargeSupp}
    Let $\H$ be the Drury-Arveson space on $\bB_2$. In this case, a $K$-contraction is simply a row contraction. Let $\fH$ be a Hilbert space and let $A\in \B(\fH)$ be a quasi-nilpotent operator which is not nilpotent and such that $\|A\|<1$. We consider the pure commuting row contraction $T=(A,0)$. We claim that 
    \[
     \Ann(T)=\{f\in \M(\H): f(z,0)=0 \text{ for every } z\in \bD\}.
    \]
	To see the non-trivial inclusion, fix $f\in \Ann(T)$. Define $f_1:\bD\to \bC$ as
	\[
	 f_1(z)=f(z,0), \quad z\in \bD.
	\]
	Assume that $f_1$ is not the zero function. Clearly, we have that $f_1$ is a bounded holomorphic function on $\bD$. Thus, there is a non-negative integer $p$ and a bounded holomorphic function $g$ on $\bD$ with $g(0)\neq 0$ such that $f_1=x^p g$. Since the spectrum of $A$ consists only of $\{0\}$, by the spectral mapping theorem we infer that $g(A)$ is invertible. We observe that
	\[
	 0=f(T)=f_1(A)=A^p g(A)
	\]
	and thus infer that $A^p=0$, contrary to the assumption that $A$ is not nilpotent. Thus, $f_1$ is the zero function and the claim is established. By virtue of Theorem \ref{T:suppEqAZ} and the fact that $x_2\in \Ann(T)$, we see that 
	$
	 \supp(\Ann(T))=\ol{\bD}\times \{0\}.
	$
	On the other hand, using that $A$ is quasinilpotent it is readily seen that $\wSpec(T)=\{(0,0)\}$ so indeed the inclusion
	\[
	 \sigma'(T)\subset \supp(\Ann (T))
	\]
is strict.
	\qed
\end{example}

The previous example shows that the inclusion $\sigma'(T)\subset \supp(\Ann (T))$ is strict in general. To show that the inclusion $\Loc(\Ann (T))\subset\sigma'(T) $ is also typically strict, we start with a general result of independent interest.


\begin{theorem}\label{T:PrimeIdealEmptyLoc} 
Let $\H$ be a regular unitarily invariant space and let $\frk{a}$ be a weak-$*$ closed prime ideal of $\MH$. 
	If $\ZBd(\frk{a})$ has positive dimension as a complex analytic variety, then $\Loc(\frk{a})$ is empty.
\end{theorem}
\begin{proof}
	Assume that $\ZBd(\frk{a})$ has dimension $m$ as a complex analytic variety. Then, it contains a biholomorphic image of $\bB_m$, and thus so does $\supp(\frk{a})$ by Theorem \ref{T:suppprop}. In particular, there is $r>0$ small enough such that $\supp(\fa)$ is not contained in $B(z,r)$ for any $z\in \bC^d$.
	Let $\frk{b},\frk{c}\subset\MH$ be weak-$*$ closed ideals such that $\frk{b}\frk{c}\subset \frk{a}$ and $\frk{c} \not\subset \frk{a}$.
	Fix $f\in\frk{c}\bksl\frk{a}$, and note that $f\frk{b}\subset\frk{a}$.
	Because $\frk{a}$ is prime and $f\nin\frk{a}$, it follows that $\frk{b}\subset\frk{a}$ and thus
	$
		\supp(\frk{a})\cap\supp(\frk{b})=\supp(\frk{a})
	$
	is not contained in $B(z,r)$ for any $z\in \bC^d$. Hence, $\supp(\frk{a})$ cannot be localized at any point.
\end{proof}

We are not aware of a description of all the prime ideals of $\mc{M}(\mc{H})$. Nevertheless, we can show that these are at least as plentiful as the prime homogeneous ideals of $\CC[x_1,\ldots,x_d]$.

\begin{lemma}
Let $\H$ be a regular unitarily invariant space and let $\frk{p}$ be a homogeneous prime ideal of $\CC[x_1,\ldots,x_d]$. Let $\frk{a}\subset \M(\H)$ be the weak-$*$ closure of $\frk{p}$.
	Then $\frk{a}$ is a prime ideal of $\MH$.
	\label{L:PHIPC} 
\end{lemma}
\begin{proof}
    It is well known that $\H=\bigoplus_{n=0}^\infty \H_n$ where $\H_n$ denotes the subspace of homogeneous polynomials of degree $n$ in $\H$ (see \cite[Example 6.1]{hartz2017isom}). For each $n\geq 0$, we let $Q_n\in \B(\H)$ denote the orthogonal projection onto $\H_n$. Because $\frk{p}$ is homogeneous, we have that $Q_n \frk{a}\subset  \frk{a}$ for every $n\geq 0$.    
	
	We now claim that 
	\begin{equation}\label{Eq:pa}
	\frk{p}=[\frk{a}\mc{H}]\cap\CC[x_1,\ldots,x_d].
	\end{equation}
	Plainly $\frk{p}\subset [\frk{a}\mc{H}]\cap\CC[x_1,\ldots,x_d]$. To see the reverse inclusion, 
	let $q\in [\frk{a}\mc{H}]\cap\CC[x_1,\ldots,x_d]$. 
    	For $N$ sufficiently large, we see that $R_N q=q$ where $R_N=\sum_{n=0}^N Q_n$. Since the range of $R_N$ is finite-dimensional we necessarily have that $R_N \frk{p}\mc{H}=R_N[\frk{a}\mc{H}].$ Moreover, it is readily verified that $R_N\frk{p}\mc{H}\subset\frk{p}$, and therefore
	\[ q=R_Nq\in R_N[\frk{a}\mc{H}]=R_N\frk{p}\mc{H}\subset \frk{p}, \]
	proving \eqref{Eq:pa}.

    Going back to proving that $\fa$ is prime in $\M(\H)$, we let $f,g\in\MH$ be such that $fg\in\frk{a}$ and $g\nin\frk{a}$. We must show that $f\in \fa$.  Decompose $f,g$ in homogeneous components as $f=\sum_{n=0}^\infty f_n$ and $g=\sum_{n=0}^\infty g_n$. 	Since $g\notin \fa$, Theorem \ref{T:LatIsomThm} implies that $\ran M_g\not\subset [\fa \H]$, and thus $g\notin[\fa \H]$. In particular, there is an index $n\geq 0$ such that $g_n\notin [\fa\H]$.
    Let $n_0\geq 0$ be the first such index and put $g'=g-\sum_{n<n_0}g_n$. It is easy to check that $g'f\in \fa$.
	
	Let $m\geq 0$ be an integer. By an application of Theorem \ref{T:GleasonTrick}, we find multipliers $\phi,\psi\in\MH$, both vanishing to order $m+n_0+1$ at the origin, such that
	\[ g'=\sum_{r=0}^{m} g_{n_0+r} + \phi, \quad f=\sum_{s=0}^{m+n_0} f_s + \psi. \]
	A calculation then reveals that
	\[
	 \sum_{j=0}^{m} g_{n_0+j}f_{m-j}=Q_{m+n_0}(g'f)\in Q_{m+n_0}(\fa)\subset \fa
	\]
	so that 
	\begin{equation}\label{Eq:prime}
	 \sum_{j=0}^{m} g_{n_0+j}f_{m-j} \in [\fa\H]\cap \bC[x_1,\ldots,x_d]=\fp
	\end{equation}
	by virtue of \eqref{Eq:pa}.
	
	For $m=0$, \eqref{Eq:prime} says that $g_{n_0}f_0\in\frk{p}$.
	Because $g_{n_0}\nin [\fa \H]$, we have that $g_{n_0}\notin\frk{p}$, so that $f_0\in\frk{p}$ because $\fp$ is prime. 
	Suppose now that we have $f_0,\dots,f_m\in\frk{p}$ for some $m\geq 0$. Using \eqref{Eq:prime} again we find 
	\[
	g_{n_0}f_{m+1}+\sum_{j=1}^{m+1}g_{n_0+j}f_{m+1-j}=\sum_{j=0}^{m+1} g_{n_0+j}f_{m+1-j}\in\frk{p}
    \]
and thus $g_{n_0}f_{m+1}\in\frk{p}$. Again, because $\frk{p}$ is prime and $g_{n_0}\nin\frk{p}$, it follows that $f_{m+1}\in\frk{p}$. By induction, we conclude that $f_n\in \frk{p}$ for every $n\geq 0$, and thus $f\in [\fa \H]$. Theorem \ref{T:LatIsomThm} implies that $f\in \fa$ and we conclude that $\frk{a}$ is prime.
\end{proof}

We can now give an example that shows that the inclusion $\Loc(\Ann(T))\subset \sigma'(T)$ is strict in general, even for the functional model.

\begin{example}\label{Ex:strictincl2}
Let $\H$ be a regular unitarily invariant space. Let $\frk{p}$ be a prime homogeneous ideal of $\CC[x_1,\dots,x_d]$ for which $\ZBd(\frk{p})$ has positive dimension as a complex analytic variety. Let $\frk{a}$ be the weak-$*$ closure of $\frk{p}$ in $\MH$. Then, $\ZBd(\frk{p})=\ZBd(\frk{a})$. By Lemma \ref{L:PHIPC}, $\frk{a}$ is a prime ideal of $\MH$, so that $\Loc(\fa)$ is empty by Theorem \ref{T:PrimeIdealEmptyLoc}. On the other hand, $\fa$ is clearly a proper ideal so that $\H(\fa)$ is not the zero space by virtue of Theorem \ref{T:LatIsomThm}. Hence $\sigma'(\kZa)$ is non-empty and the inclusion $\Loc(\fa)\subset \sigma'(\kZa)$ is strict. 
\qed
\end{example}

The foregoing discussion seems to indicate that $\Loc(\Ann(T))$ can contribute information about the spectrum only where $\ZBd(\Ann(T))$ is  zero-dimensional. This intuition will guide us for the remainder of the paper, wherein we extract spectral information whenever the zero set of the annihilator is small. A key technical tool will be the following.

\begin{theorem}
Let $\H$ be a regular unitarily invariant space with kernel $K$. Let $T=(T_1,\dots,T_d)$ be a $K$-pure commuting $K$-contraction and let $\fa=\Ann(T)$. Let $C\subset \supp(\fa)$ be a non-empty compact subset such that $\supp(\fa)\bksl C$ is also compact. Then $C\cap\TSpec(T)$ is non-empty.
	\label{T:KNotEmptyIntersectNotEmpty}
\end{theorem}
\begin{proof} 
	Let $\frk{H}$ denote the Hilbert space on which $T_1,\ldots,T_d$ act. By Lemma \ref{L:wSpecInSupp}, we have that $\TSpec(T)\subset \supp(\frk{a})$. In particular, we may assume that $\supp(\fa)\neq C$, so that $\supp(\fa)\setminus C$ is non-empty. 	
	Since $C$ and $\supp(\frk{a})\bksl C$ are disjoint non-empty compact sets, there are disjoint open subsets $U$ and $V$ of $\bC^d$ such that $C\subset U$ and $ \supp(\frk{a})\bksl C\subset V$. Thus \[\TSpec(T)\subset \supp(\fa)\subset U\cup V\] and $\chi_U$, the characteristic function of $U$, is analytic on a neighborhood of $\TSpec(T)$.

	We claim that $\chi_U(T)$ is non-zero. To see this, we argue by contradiction and assume that $\chi_U(T)=0$. We apply Corollary \ref{C:GetInter} to see that there exist a Hilbert space $\fE$ and an isometry $V:\fH\to\H(\fa)\otimes \fE$ such that 
    \[
     VT^*=((\kZa)^*\otimes I_\fE)V.
     \]
	Moreover, $\ker V^*$ contains no non-zero closed subspace of the form $\frk{M}\otimes\frk{E}$ where $\fM\subset \H(\fa)$ is invariant for $\kZa$. Next, we invoke Theorem \ref{T:suppEqSpec} to see that $\supp(\fa)=\TSpec(\kZa)$. Thus, $\supp(\fa)=\TSpec(\kZa\otimes I_\fE)$ as seen in part (iv) of Theorem \ref{T:Taylorprop}. Hence, $\chi_U(\kZa\otimes I_\fE)$ and $\chi_U(\kZa)$ are well defined. By parts (ii) and (v) of Theorem \ref{T:Taylorprop}, we infer that
	\[ \chi_U(T)V^*=V^* \chi_U(\kZa\otimes I_{\frk{E}})= V^* ( \chi_U(\kZa)\otimes I_{\frk{E}}). \]
	Since $\chi_U(T)=0$ we must have $\ran(\chi_U(\kZa)\otimes I_{\frk{E}})\subset\ker V^*$.
	Moreover, $C\subset U$ and thus
	\[ C=\TSpec(\kZa|_{\ran\chi_U(\kZa)}) \]
	by part (iii) of Theorem \ref{T:Taylorprop}.
	Since $C$ is non-empty, we infer that $\ran \chi_U(\kZa)$ is a non-zero closed subspace which is invariant for $\kZa$. Thus, the inclusion
	\[(\ran \chi_U(\kZa))\otimes\frk{E}= \ran(\chi_U(\kZa)\otimes I_{\frk{E}}) \subset \ker V^*\]
	contradicts the given property of $V$. The claim is established.

	If $C\cap\TSpec(T)$ were empty, then $V$ would be an open neighbourhood of $\TSpec(T)$ on which $\chi_U$ vanishes identically, so that $\chi_U(T)=0$, contrary to what was proved in the previous paragraph. 
	\end{proof}

Our next goal is to show a weak-$*$ closed ideal can be localized at any of its isolated points. To achieve this, we start with an ideal-theoretic analogue of part (iii) of Theorem \ref{T:Taylorprop}.

\begin{theorem}
	Let $\H$ be a maximal regular unitarily invariant space and let $\frk{a}$ be a weak-$*$ closed ideal of $\MH$.
	If $C_1,C_2$ are disjoint non-empty compact subsets of $\supp(\frk{a})$ such that $\supp(\frk{a})=C_1\cup C_2$, then there exist weak-$*$ closed ideals $\frk{a}_1,\frk{a}_2$ in $\MH$ such that
	\[ \frk{a}=\frk{a}_1\cap\frk{a}_2, \quad \supp(\frk{a}_1)=C_1, \quad \supp(\frk{a}_2)=C_2. \]
	\label{T:FactorOverCompacts}
\end{theorem}
\begin{proof}
    Choose disjoint open sets $U_1$ and $U_2$ such that $C_1\subset U_1$ and $C_2\subset U_2$. Put $Q_1=\chi_{U_1}(\kZa)$ and $Q_2=\chi_{U_2}(\kZa)$. Recall now that $\kZa$ is $K$-asymptotically vanishing, and thus so are the commuting $d$-tuples $R=\kZa|_{\ran Q_1}$ and $S=\kZa|_{\ran Q_2}$. By Theorem \ref{T:pureKcontr}, we infer that $R$ and $S$ are $K$-pure. Set $\fa_1=\Ann(R)$ and $\fa_2=\Ann(S)$. It is clear that $\frk{a}\subset\fa_1\cap\fa_2$,
	so by Theorem \ref{T:suppprop} we may infer
	\begin{equation}\label{Eq:incl1} \supp(\fa_1)\cup\supp(\fa_2)\subset \supp(\fa)= C_1\cup C_2. \end{equation}
	 We see by Theorem \ref{T:suppEqSpec} and part (iii) of Theorem \ref{T:Taylorprop} that $\TSpec(R)=C_1$ and $\TSpec(S)=C_2$. In turn,  Lemma \ref{L:wSpecInSupp} implies that
	 \begin{equation}\label{Eq:incl2}
	  C_1=\TSpec(R)\subset \supp(\fa_1).
	 \end{equation}
    Setting $D_2=\supp(\fa_1)\cap C_2$, it is clear that $C_1$ and $D_2$ are disjoint compact sets while \eqref{Eq:incl1} and \eqref{Eq:incl2} imply that $\supp(\fa_1)=C_1\cup D_2$.
	Because \[\TSpec(R)\cap D_2\subset C_1\cap C_2=\varnothing\] it follows from Theorem \ref{T:KNotEmptyIntersectNotEmpty} that $D_2$ is empty, whence $\supp(\fa_1)=C_1$.
	In perfectly analogous fashion, one sees that $\supp(\fa_2)=C_2$.
	Note now that $I=Q_1+Q_2$ and $Q_1Q_2=Q_2Q_1=0$, so a routine argument shows that  $\kZa$ is similar to $R\oplus S$.
	It is readily seen that the similarity of $\kZa$ and $R\oplus S$ implies the similarity of $f(\kZa)$ and $f(R)\oplus f(S)$ for every $f\in\M(\H)$, and therefore
	\[ \frk{a}=\Ann(\kZa)=\Ann(R\oplus S)=\fa_1\cap\fa_2. \qedhere \]
\end{proof}

The following consequence is not needed for our purposes, but it is of independent interest. It says that removing one isolated point from a zero set for $\M(\H)$ yields another zero set. 

\begin{corollary}\label{C:ZeroSetsMinusPoints}
	Let $\H$ be a maximal regular unitarily invariant space and let $\fa$ be a weak-$*$ closed ideal of $\M(\H)$. Suppose $z$ is an isolated point of $\Z_{\bB_d}(\fa)$. Then, there exists another weak-$*$ closed ideal $\fb$ of $\M(\H)$ such that
	$
	 \Z_{\bB_d}(\fb)=\Z_{\bB_d}(\fa)\setminus \{z\}.
	$
\end{corollary}
\begin{proof}
    
	If $\Z_{\bB_d}(\fa)=\{z\}$, then 
	\[\Z_	 {\bB_d}(\fa)\bksl\{z\}=\varnothing=\Z_{\bB_d}(\M(\H)).\]
	Suppose now that $\Z_{\bB_d}(\fa)\bksl\{z\}$ is non-empty.
	The singleton $\{z\}$ is a compact subset of $\supp(\frk{a})$ such that $\supp(\frk{a})\bksl\{z\}$ is also compact. Apply Theorem \ref{T:FactorOverCompacts} to find weak-$*$ closed ideals $\frk{b},\frk{c}$ of $\M(\H)$ such that $\frk{a}=\frk{b}\cap\frk{c}$, $\supp(\frk{b})=\supp(\frk{a})\bksl\{z\}$ and $\supp(\frk{c})=\{z\}$.
	Invoking Theorem \ref{T:suppprop}, we find \[\ZBd(\frk{b})=\BB_d\cap(\supp(\frk{a})\bksl\{z\})=\Z_{\bB_d}(\fa)\bksl\{z\}.\]
	as desired.
\end{proof}

We remark that in the case where $\H$ is the Hardy space, the previous result is an easy consequence of the known characterization of the zero sets for $H^\infty(\bD)$ in terms of the Blaschke condition \cite[page 64]{hoffman1988}. No such characterization appears to be known in general.

Going back to the task at hand, we can show that an ideal can be localized at every isolated point of its support. 
\begin{corollary}
Let $\H$ be a maximal regular unitarily invariant space and let $\fa$ be a weak-$*$ closed ideal of $\M(\H)$. Suppose $z\in\supp(\frk{a})$ has the property that for every $r>0$ there exists an open set $U\subset B(z,r)$ containing $z$ such that $\supp(\frk{a})\cap U$ is both closed and open in $\supp(\frk{a})$.
	Then, $z\in\mathrm{Loc}(\frk{a})$. In particular, $\Loc(\fa)$ contains all the isolated points of $\supp(\fa)$.
	\label{C:ClopenLem}
\end{corollary}
\begin{proof}
    If $\supp(\fa)=\{z\}$, then it is trivial that $z\in \Loc(\fa)$. Assume therefore that $\supp(\fa)\setminus\{z\}$ is non-empty. Hence, there is $r_0>0$ small enough so that $\supp(\fa)$ is not contained in $B(z,r_0)$. Let $0<r<r_0$ and let $U\subset B(z,r)$ be an open set containing $z$ such that $\supp(\frk{a})\cap U$ is open and closed in $\supp(\frk{a})$.
	Since $\supp(\frk{a})$ is compact, it follows that both $\supp(\frk{a})\cap U$ and $\supp(\frk{a})\bksl U$ are compact and non-empty.
	By Theorem \ref{T:FactorOverCompacts}, there are weak-$*$ closed ideals $\frk{b},\frk{c}$ of $\MH$ such that $\supp(\frk{b})=\supp(\frk{a})\cap U$, $\supp(\frk{c})=\supp(\frk{a})\bksl U$, and $\frk{a}=\frk{b}\cap\frk{c}$.
	Since $z\nin\supp(\frk{c})$, it follows from Theorem \ref{T:suppprop} that $\frk{c} \not\subset\frk{a}$.
	Plainly $\frk{b}\frk{c}\subset\frk{a}$ and
	\[ \supp(\frk{a})\cap\supp(\frk{b})=\supp(\frk{a})\cap U\subset B(z,r). \]
	Thus $z\in\Loc(\frk{a})$.
\end{proof}

Using this information, we can complement Theorem \ref{T:LocInSigPrime} in some special cases. 

\begin{corollary}\label{C:LocDimZero}
Let $\H$ be a maximal regular unitarily invariant space with kernel $K$.
	If $T=(T_1,\dots,T_d)$ is an absolutely continuous commuting $K$-contraction such that $\ZBd(\Ann (T))$ is discrete, then
	\[ \Loc(\Ann (T))\cap\BB_d=\sigma'(T)\cap\BB_d=\wSpec(T)\cap\BB_d=\ZBd(\Ann (T)). \]
\end{corollary}
\begin{proof}
 By Theorem \ref{T:LocInSigPrime}, we see that $\Loc(\Ann(T))\subset \sigma'(T)$. Furthermore, we know from Lemma \ref{L:wSpecInSupp} that $\wSpec(T)\subset \supp(\Ann(T))$. Thus, we find
\[ \Loc(\Ann (T))\cap\BB_d\subset \sigma'(T)\cap\BB_d\subset \wSpec(T)\cap\BB_d\subset \supp(\Ann(T))\cap \bB_d. \]
By Theorem \ref{T:suppprop}, we see that $\supp(\Ann(T))\cap \bB_d=\Z_{\bB_d}(\Ann(T))$ is discrete, so all its points are isolated. Corollary \ref{C:ClopenLem} then gives
\[
 \Loc(\Ann (T))\cap\BB_d=\supp(\Ann(T))\cap \bB_d. \qedhere
\] 
\end{proof}

When the support of the annihilator is totally disconnected, the preceding corollary can be extended to apply to the sphere as well.

\begin{theorem}
Let $\H$ be a maximal regular unitarily invariant space with kernel $K$. 	Suppose $T=(T_1,\dots,T_d)$ is an absolutely continuous commuting $K$-contraction. If $\supp(\Ann (T))$ is totally disconnected, 
	then 
	\[ \Loc(\Ann (T))=\sigma'(T)=\wSpec(T)=\supp(\Ann (T)). \]
	\label{T:TotDiscon}
\end{theorem}
\begin{proof}
	From Lemma \ref{L:wSpecInSupp} and Theorem \ref{T:LocInSigPrime}  it follows that
	\[ \Loc(\Ann (T))\subset \sigma'(T)\subset\wSpec(T)\subset \supp(\Ann (T)). \]
		Because $\supp(\Ann (T))$ is totally disconnected, compact and Hausdorff, it has a base for its topology consisting of sets that are simultaneously closed and open \cite[Theorem 29.7]{Willard_GenTop}. Hence, for every $z\in\supp(\Ann (T))$ and $r>0$ there exists an open subset $U\subset \bC^d$ containing $z$ such that $U\subset B(z,r)$ and $U\cap\supp(\Ann (T))$ is closed and open in $\supp(\Ann (T))$.
	By Corollary \ref{C:ClopenLem}, it follows that $\supp(\Ann (T))\subset\mathrm{Loc}(\Ann T)$.
\end{proof}

\subsection{The Taylor spectrum of $K$-pure commuting $K$-contractions}
We close the paper by leveraging the tools developed above to study the Taylor spectrum of $K$-pure commuting $K$-contractions. We start with a simple fact.

\begin{corollary}\label{C:IsoInTSpec}
Let $\H$ be a regular unitarily invariant space with kernel $K$. Let $T=(T_1,\dots,T_d)$ be a $K$-pure commuting $K$-contraction. Then, the isolated points of $\supp(\Ann(T))$ all lie in $\TSpec(T)$. 
\end{corollary}
\begin{proof}
Simply apply Theorem \ref{T:KNotEmptyIntersectNotEmpty} with any isolated point playing the role of the compact set.
\end{proof}

The previous statement can be significantly refined in the case where the annihilator has a discrete zero set. Before proceeding, we require the following lemma.
\begin{lemma}
    Let $\H$ be a regular unitarily invariant space with kernel $K$.
	Let $T=(T_1,\dots,T_d)$ be an absolutely continuous commuting $K$-contraction. Let $z$ be an isolated point of $\ZBd(\Ann (T))$ such that $\TSpec(T)=\{z\}$. Then, there is a $d$-tuple $N$ of commuting nilpotent operators such that $T=zI+N.$
	\label{L:NilpotentLem}
\end{lemma}
\begin{proof}
The proof of \cite[Theorem 4.3]{CTInterp} adapts verbatim.
\end{proof}

We now state our desired result. 

\begin{theorem}\label{T:PureDimZero}
Let $\H$ be a maximal regular unitarily invariant space with kernel $K$.
	Let $T=(T_1,\dots,T_d)$ be a $K$-pure commuting $K$-contraction.
Assume that $\Z_{\bB_d}(\Ann(T))$ is discrete. Then,
	\[ \sigma_{\mathrm{p}}(T)=\TSpec(T)\cap\BB_d=\ZBd(\Ann (T)). \]
	Furthermore, if $\supp(\Ann (T))$ is totally disconnected, then $\TSpec(T)=\supp(\Ann (T))$.
\end{theorem}
\begin{proof}
	Denote by $\frk{H}$ the Hilbert space on which $T_1,\ldots,T_d$ act. Put $\fa=\Ann(T)$.
	By \eqref{Eq:incTayComm}, \eqref{Eq:incTayPoint}, Theorem \ref{T:suppprop} and Lemma \ref{L:wSpecInSupp}, we have
	\begin{equation}\label{Eq:inclspec} \sigma_{\mathrm{p}}(T)\cap\BB_d\subset \TSpec(T)\cap\BB_d\subset\ZBd(\fa)\subset \supp(\fa). \end{equation}
	Fix $z\in\ZBd(\fa)$. Because $\ZBd(\fa)$ is assumed to be discrete, we see that $z$ is an isolated point of $\supp(\fa)$, so $z\in\TSpec(T)$ by Corollary \ref{C:IsoInTSpec}. 
	
	Next, we claim that there is a closed $T$-invariant subspace $\fX_1\subset \fH$ such that $\TSpec(T|_{\frk{X}_1})=\{z\}$. If $\TSpec(T)=\{z\}$ we may simply choose $\fX_1=\fH$, so we assume that $\TSpec(T)\setminus\{z\}$ is non-empty. It follows from part (iii) of Theorem \ref{T:Taylorprop} that there are $T$-invariant subspaces $\frk{X}_1,\frk{X}_2$ such that $\frk{X}_1\cap\frk{X}_2=\{0\}$, $\frk{X}_1+\frk{X}_2=\frk{H}$, and
	\[ \TSpec(T|_{\frk{X}_1})=\{z\}, \quad \TSpec(T|_{\frk{X}_2})=\TSpec(T)\bksl\{z\}. \]
	Clearly, we have $\fa\subset \Ann (T|_{\mc{X}_1})$ so that $\ZBd(\Ann (T|_{\mc{X}_1}))\subset \ZBd(\fa)$. Arguing as above, we know that 
	\[
	 \TSpec(T|_{\mc{X}_1})\cap\BB_d\subset \ZBd(\Ann (T|_{\mc{X}_1}))
	\]
    so $z$ is an isolated point of $\ZBd(\Ann (T|_{\mc{X}_1}))$.
	By Lemma \ref{L:NilpotentLem}, we infer that $T|_{\frk{X}_1}=zI+N$ for a $d$-tuple $N$ of commuting nilpotent operators on $\frk{X}_1$.
	Choose $\alpha\in\NN^d$ with maximal length such that $N^\alpha\neq 0$. Then, $N_jN^\alpha=0$ for $j=1,\dots,d$.
	There exists $h\in\frk{X}_1$ such that $N^\alpha h\neq 0$, and so
	\[ T_jN^\alpha h=(z_jI+N_j)N^\alpha h=z_j N^\alpha h. \]
	This shows that $z\in\sigma_{\mathrm{p}}(T)$. We conclude from \eqref{Eq:inclspec} that
	\[ \sigma_{\mathrm{p}}(T)\cap\BB_d=\TSpec(T)\cap\BB_d=\ZBd(\fa). \]
    In view of Lemma \ref{L:Kcontevalue}, we see that $\pSpec(T)\subset \bB_d$ so in fact we have
    \[ \sigma_{\mathrm{p}}(T)=\TSpec(T)\cap\BB_d=\ZBd(\fa). \]

	Finally, assume that $\supp(\fa)$ is totally disconnected and that $z$ lies in $\supp(\fa)$. If $\supp(\fa)=\{z\}$, then we see that $\TSpec(T)=\{z\}$ by Lemma \ref{L:wSpecInSupp}. Assume now that $\supp(\fa)\setminus\{z\}$ is non-empty.	Hence, there is $r_0>0$ small enough so that $\supp(\fa)$ is not contained in $B(z,r_0)$. Let $0<r<r_0$. As in the proof of Theorem \ref{T:TotDiscon}, there is an open set $U\subset B(z,r)$  containing $z$ such that $\supp(\frk{a})\cap U$ is open and closed in $\supp(\frk{a})$. Put  $C= \supp(\frk{a})\cap U$. Then, $C$ and $\supp(\frk{a})\bksl C$ are non-empty, disjoint and compact.
	It follows from Theorem \ref{T:KNotEmptyIntersectNotEmpty} that $C\cap \TSpec(T)$ is non-empty.
	Thus $B(z,r)\cap \TSpec(T)$ is non-empty for every $0<r<r_0$, and so $z\in \TSpec(T)$ as $\TSpec(T)$ is closed. We conclude that $\supp(\frk{a})\subset \TSpec(T)$, and Lemma \ref{L:wSpecInSupp} shows that in fact equality holds.
\end{proof}

We end with a finite-dimensional example in the case where $\H$ is the Drury--Arveson space.

\begin{example}
	Let $T_1,\dots,T_d$ be commuting  $n\times n$ complex matrices such that $T_1T_1^*+\dots+T_dT_d^*<I$. Thus, $T=(T_1,\ldots,T_d)$ is a pure commuting row contraction, and hence a $K$-pure commuting $K$-contraction, where $K$ is the kernel of the Drury--Arveson space. We put $\fa=\Ann(T)$. For each $j=1,\ldots,d$, we let $q_j$ denote the minimal polynomial of $T_j$. We define $r_1,\ldots,r_d\in \bC[x_1,\ldots,x_d]$ as $r_j=q_j(x_j)$ for every $j=1,\ldots,d$. Then, we find $r_1,\dots,r_d\in\fa$, whence \[\supp(\fa)\subset \ZcBd(r_1)\cap\dots\cap\ZcBd(r_d)\] by Corollary \ref{C:AHDense}.
	Thus $\supp(\fa)$ is finite so
	\[\TSpec(T)\cap\BB_d=\sigma_{\mathrm{p}}(T)=\ZBd(\fa) \qand \TSpec(T)=\supp(\fa)\] by Theorem \ref{T:PureDimZero}.
	But $T_1T_1^*+\dots+T_dT_d^*<I$ so $\TSpec(T)\subset \wSpec(T)\subset \bB_d$ and thus
	\[ \TSpec(T)=\sigma_{\mathrm{p}}(T)=\ZBd(\fa) \]
	is finite.	By part (iii) of Theorem \ref{T:Taylorprop} and Lemma \ref{L:NilpotentLem}, we see that $T$ is similar to $D=\bigoplus_{z\in\ZBd(\fa)}(zI+N^{(z)})$ where $N^{(z)}$ is a commuting $d$-tuple of nilpotent operators for each $z\in\ZBd(\fa)$.
	
	We claim that $\ZBd(\fa)=\ZBd(\fa\cap\CC[x_1,\ldots,x_d])$, so that the polynomial part of $\fa$ suffices to determine the spectrum of $T$, as one would expect.
	It suffices to show that the polynomial ideal $\frk{p}=\frk{a}\cap\CC[x_1,\ldots,x_d]$ is weak-$*$ dense in $\frk{a}$.
	Note that there exists a positive integer $k$ such that any analytic function $g$ that vanishes to at least order $k$ at each point of $\ZBd(\fa)$ must satisfy $g(D)=0$, and hence $g(T)=0$ by part (ii) of Theorem \ref{T:Taylorprop}.
	With $\frk{b}$ denoting the ideal of those functions in $\M(\H)$ that vanish to at least order $k$ on $\ZBd(\fa)$, we have $\frk{b}\subset \frk{a}$.
	An application of Theorem \ref{T:GleasonTrick} shows that $\fb$ is generated by polynomials and thus $\fb$ is contained in the weak-$*$ closure of $\fp$.
	For any $f\in\fa$ there is a polynomial $p$ such that $f-p\in\fb$, 
	and so for such $p$ we have $p(T)=0$.
	Thus $\fa\subset \fp+\fb$, and $\fa$ is the weak-$*$ closure of $\fp$.
	\qed
\end{example}

\bibliographystyle{plain}
\bibliography{biblio_main.bib}

\end{document}